\documentclass[12pt]{amsart}%

\usepackage{graphicx}
\usepackage{multicol,multirow}
\usepackage{amsmath,amssymb,amsfonts}
\usepackage{mathrsfs}
\usepackage{amsthm}
\usepackage{rotating}
\usepackage{appendix}
\usepackage[numbers]{natbib}
\usepackage{ifpdf}
\usepackage[T1]{fontenc}
\usepackage{newtxtext}
\usepackage{newtxmath}
\usepackage{textcomp}
\usepackage[colorlinks, allcolors=teal]{hyperref}

\usepackage{amssymb, amsthm, amsmath, amsfonts, amsxtra, mathrsfs, mathtools}
\usepackage{color}
\usepackage{import}
\usepackage{thmtools}
\usepackage{thm-restate}

\usepackage[intoc]{nomencl}
\makenomenclature
\usepackage{stmaryrd}
\usepackage{svg}
\usepackage{tikz-cd}
\usepackage{tikz}
\usepackage{tikzsymbols}
\usetikzlibrary{decorations.pathreplacing,angles,quotes}
\usetikzlibrary{patterns}

\usepackage{enumitem}
\usetikzlibrary{matrix}
\usepackage[ruled,vlined]{algorithm2e}
\usepackage[mathscr]{euscript}
\usepackage[normalem]{ulem}
\usepackage{comment}
\usepackage{multicol}
\usepackage{caption, subcaption}
\input xy
\xyoption{all}

\newtheorem{theorem}{Theorem}[section]
\newtheorem*{theorem*}{Theorem}
\newtheorem*{thm:wcproduct}{Theorem \ref{thm:wcproduct}}
\newtheorem*{thm:LrnChow}{Theorem \ref{thm:LrnChow}}
\newtheorem{proposition}[theorem]{Proposition}
\newtheorem{lemma}[theorem]{Lemma}

\newtheorem{conjecture}[theorem]{Conjecture}
\theoremstyle{definition}
\newtheorem{definition}[theorem]{Definition}

\newtheorem{remark}[theorem]{Remark}
\newtheorem{example}[theorem]{Example}

\def\Z{\mathbb{Z}}
\def\Q{\mathbb{Q}}
\def\R{\mathbb{R}}
\def\C{\mathbb{C}}
\def\P{\mathbb{P}}

\def\Z{\mathbb{Z}}

\def\calA{\mathcal{A}}

\def\calG{\mathcal{G}}
\def\calH{\mathcal{H}}

\def\calL{\mathcal{L}}

\def\calV{\mathcal{V}}

\def\tbI{\widetilde{\mathbf{I}}}
\def\tbJ{\widetilde{\mathbf{J}}}
\def\tI{\widetilde{I}}
\newcommand{\w}{\mathbf{w}}

\def\a{\mathfrak{a}}

\def\b{\mathfrak{b}}

\def\L{\overline{\mathcal{L}}}
\def\M{\overline{\mathcal{M}}}

\def\Lrntrop{L^{r,\text{trop}}_n}
\newcommand{\Lrn}{\L^r_n}

\title{Wonderful compactifications and rational curves with cyclic action}

\author[E. Clader]{Emily Clader}\address{Emily Clader, Department of Mathematics, San Francisco State University}
\email{\url{eclader@sfsu.edu}}

\author[C. Damiolini]{Chiara Damiolini}
\address{Chiara Damiolini, Department of Mathematics, University of Texas at Austin}
\email{\url{chiara.damiolini@austin.utexas.edu}}

\author[S. Li]{Shiyue Li}\address{Shiyue Li, School of Mathematics, Institute for Advanced Study}
\email{\url{shiyue_li@ias.edu}}

\author[R. Ramadas]{Rohini Ramadas}\address{Rohini Ramadas, Department of Mathematics, Warwick Mathematics Institute}
\email{\url{Rohini.Ramadas@warwick.ac.uk}}

\begin{document}

\begin{abstract}We prove that the moduli space of rational curves with cyclic action, constructed in our previous work, is realizable as a wonderful compactification of the complement of a hyperplane arrangement in a product of projective spaces.  By proving a general result on such wonderful compactifications, we conclude that this moduli space is Chow-equivalent to an explicit toric variety (whose fan can be understood as a tropical version of the moduli space), from which a computation of its Chow ring follows. 
\end{abstract}

\maketitle

\section{Introduction}

The moduli space $\L^r_n$ of rational curves with cyclic action was constructed in our previous work \cite{CDHLR} as a generalization of Losev and Manin's moduli space of rational curves with weighted marked points.  In particular, the Losev--Manin space $\L_n$, introduced in \cite{LM}, is a toric variety whose associated polytope is the permutohedron $\Pi_n$, and the torus-invariant subvarieties of $\L_n$ have a modular interpretation as ``boundary strata,'' so one obtains an inclusion- and dimension-preserving bijection between the boundary strata of $\L_n$ and the faces of $\Pi_n$.  This work was generalized by Batyrev and Blume, who in \cite{BB} constructed a toric moduli space $\L^2_n$ of rational curves with involution whose boundary strata are encoded by the faces of the signed permutohedron.  Generalizing the story further, the moduli space $\L^r_n$ parameterizes certain rational curves with an automorphism of order $r$ and weighted orbits.  Although $\L^r_n$ is not toric when $r > 2$, its boundary strata are nevertheless encoded by a polyhedral object: not a polytope, in this case, but a polytopal complex.  In this way, $\L^r_n$ appears to occupy an intriguing middle ground between toric varieties and more general moduli spaces of rational curves.

The goal of the current work is to realize $\L^r_n$ as a wonderful compactification of the complement of a particular arrangement of hyperplanes in $(\P^1)^n$, and in doing so, to give a combinatorial description of its Chow ring.  Wonderful compactifications were introduced by De Concini and Procesi in \cite{DCP} as a way to compactify the complement of an arrangement of hyperplanes in $\P^n$ so that much of the geometry of the compactification is encoded in the combinatorics of the original hyperplane arrangement. The geometry of these spaces has been used to resolve long-standing conjectures in combinatorics like the log-concavity of characteristic polynomials of matroids \cite{AHK} and the Dowling--Wilson top-heavy conjecture \cite{braden2020singular}.  On the other hand, they have also provided a valuable new perspective in geometry; perhaps the most relevant example for the present work is the Deligne--Mumford--Knudsen compactification $\M_{0,n}$, which can be realized as a wonderful compactification of the braid arrangement complement in $\P^{n-3}$, from which one obtains an elegant presentation of its Chow ring.

One way in which to understand the Chow ring in this setting, as shown by Feichtner and Yuzvinsky in \cite{FY}, is as the Chow ring of the toric variety of a fan $\Sigma_{\calG}$ that can be combinatorially associated to a hyperplane arrangement in projective space together with a ``building set'' $\calG$.  In particular, the data of $\calG$ specifies a wonderful compactification $\overline{Y}_{\calG}$ of the arrangement complement, and Feichtner--Yuzvinsky prove that the Chow ring of $\overline{Y}_{\calG}$ is isomorphic to that of the toric variety $X_{\Sigma_{\calG}}$.

The construction of wonderful compactifications was generalized by Li Li in \cite{LiLi} to complements of arrangements of subvarieties in a smooth variety, but some of their combinatorial nature is lost in this generality.  In particular, the geometry of a wonderful compactification $\overline{Y}_{\calG}$ is not determined merely by the intersection combinatorics of the subvarieties in the arrangement---which is what determines $\Sigma_{\calG}$---but by the particular geometry of the subvarieties themselves.  Thus, one should not expect the Chow ring of $\overline{Y}_{\calG}$ to be isomorphic to that of a toric variety in general.

The case $\L^r_n$ of interest for our work is a wonderful compactification of a hyperplane arrangement not in a projective space (as in De Concini--Procesi's original work) but in a product of projective spaces.  Specifically, it is a ``product arrangement'' in the sense that the hyperplanes are pulled back via projection to the individual projective space factors.  We begin by proving that, for arrangements of this form, the Chow ring of the wonderful compactification is still combinatorial: one can associate a fan $\Sigma_{\calG}$ (defined in Definition~\ref{def:nestedsetfan} below) generalizing the fan of Feichtner--Yuzvinsky, and the resulting toric variety has isomorphic Chow ring to $\overline{Y}_{\calG}$. 

\begin{thm:wcproduct}[See Section~\ref{subsec:wcproduct} for precise statement] Let $\calA$ be a product arrangement in $\P^{k_1} \times \cdots \times \P^{k_n}$, let $\calG$ be a building set for its intersection lattice, and let $\Sigma_{\calG}$ be the associated nested set fan.
Then there is a Chow-equivalence 
\[A^*(\overline{Y}_{\calG}) = A^*(X_{\Sigma_{\calG}}).\]
\end{thm:wcproduct} 

Equipped with this result, we specifically consider the arrangement of hyperplanes
\[\widetilde{H}_i^j = \{(p_1, \ldots, p_n) \in (\P^1)^n \; | \; p_i = \zeta^j\}\]
for each $i \in \{1, 2, \ldots, n\}$ and each $j \in \{0,1,\ldots, r-1\}$, where $\zeta$ is a fixed $r$th root of unity.  We prove in Theorem~\ref{thm:Lrnwc} that $\L^r_n$ is the wonderful compactification of this arrangement with its maximal building set.  Denoting the associated fan by $\Sigma^r_n$, we obtain by Theorem~\ref{thm:wcproduct} an explicit computation of the Chow ring $A^*(\L^r_n)$.

To describe this computation, we first recall from \cite{CDHLR} that there is a special codimension-$1$ subvariety $D_{\tI} \subseteq \L^r_n$---specifically, a boundary divisor---associated to any ``$\mathbb{Z}_r$-decorated subset of $[n]$,'' which is a pair $\tI = (I,\a)$ in which $I \subseteq \{1, 2, \ldots, n\}$ is a nonempty set and $\a$ is a function $I \rightarrow \{0,1, \ldots, r-1\}$.  There is a partial ordering on decorated subsets given by
\[(I,\a) \leq (J,\b) \; \text{ if and only if} \; I \subseteq J \text{ and } \a(i) = \b(i)  \text{ for all } i \in I.\]
With this notation, the presentation of $A^*(\L^r_n)$ is as follows.

\begin{thm:LrnChow}
The Chow ring of $\L^r_n$ is generated by the boundary divisors $D_{\tI}$ for each (nonempty) $\Z_r$-decorated subset $\tI$ of $\{1, \dots, n\}$, with relations given by
\begin{itemize}
    \item $D_{\tI} \cdot D_{\widetilde{J}} = 0$ unless either $\tI \leq \widetilde{J}$ or $\widetilde{J} \leq \tI$;
    \item for all $i \in \{1,2,\ldots, n\}$ and all $a,b \in \{0,1,\ldots, r-1\}$,
    \[\sum_{\substack{\tI \text{ s.t. }\\ i \in I, \; \a(i) = a}} D_{\tI} = \sum_{\substack{\tI \text{ s.t. }\\ i \in I, \; \a(i) = b}} D_{\tI}.\]
\end{itemize}
\end{thm:LrnChow}

We conclude the paper by giving two other interpretations of the fan $\Sigma^r_n$, which are interesting in their own right.  First, analogously to the case of $\M_{0,n}$, we show in Proposition~\ref{prop:tropical} that this fan can be identified with a moduli space $\Lrntrop$ of ``tropical $(r,n)$-curves."  And second, analogously to the way in which the permutohedron $\Pi_n$ is the normal polytope of the fan of Losev--Manin space $\L_n$, we show in  Proposition~\ref{prop:normalfan} that the polytopal complex $\Delta^r_n$ constructed in \cite{CDHLR} is a normal complex of $\Sigma^r_n$, in the sense developed by Nathanson--Ross in \cite{NR}.  This gives a more geometric interpretation of the correspondence between the boundary strata of $\L^r_n$ and the faces of $\Delta^r_n$ that was proven combinatorially in our previous work.

Leveraging the above connection to tropical geometry, we hope in future work to use tropical intersection theory on $\Lrntrop$ to study intersection numbers on $\L^r_n$ (along the lines of \cite{katz2012tropical, kerber2009intersecting, hahn2022intersecting}). We may also study the reduced rational cohomology of the locus of tropical curves with total edge length $1$ in $\Lrntrop$ to understand the mixed Hodge structure of $\calL^r_n$, in the sense of \cite{deligne71theorie, deligne74theorie} and along the lines of \cite{chan2021tropical, kannan2020topology}.  This is made possible by the observation that the boundary $\L^{r}_n \setminus \calL^r_n$ is a divisor with simple normal crossings \cite[Observation 3.6]{CDHLR}.

\begin{remark}
    Soon after this manuscript's appearance, Eur, Fink, Larson and Spink studied the type-$B$ permutohedral toric variety $X_{B_{n}}$, which is precisely $\L^{2}_n$, in relation to delta-matroids \cite{eur2022signed}. The central combinatorial construction there is the $B_n$ permutohedral fan $\Sigma_{B_n}$, which coincides with the permutohedral fan $\Sigma^{2}_n$ constructed in the present paper. Among many things, the authors give an exceptional isomorphism $\phi^{B} \colon K(X_{B_n}) \to A(X_{B_n})$ which yields a Hirzebruch--Riemann--Roch-type theorem. Their results and techniques, together with the constructions in the present paper, will be valuable hints for the potential developments for general $\L^r_n$ discussed in Remark \ref{rem:exceptional-iso}. 
\end{remark}

\subsection*{Plan of the paper}

We begin, in Section~\ref{sec:WCbackground}, by reviewing the necessary background on wonderful compactifications and proving Theorem~\ref{thm:wcproduct}; this section is entirely self-contained, so it can be read independently by a reader interested primarily in wonderful compactifications.  In Section~\ref{sec:Lrnbackground}, we recall the definition of $\L^r_n$ and we prove that it is indeed a wonderful compactification of the arrangement in $(\P^1)^n$ described above.  Section~\ref{sec:chowring} combines these results to prove the presentation of the Chow ring in Theorem~\ref{thm:LrnChow}.  Finally, Section~\ref{sec:tropical} describes the connections both to tropical $(r,n)$-curves and normal complexes.

\subsection*{Acknowledgments}

The authors are grateful to Melody Chan, Chris Eur, Daoji Huang, Diane Maclagan, and Dustin Ross for many valuable conversations and insights, and to ICERM for hosting the ``Women in Algebraic Geometry" workshop at which this collaboration began.  The first author was supported by NSF CAREER grant 2137060. The third author was supported by the Coline M. Makepeace Fellowship from Brown University and partially supported by NSF DMS grant 1844768.

\section{Wonderful compactifications}
\label{sec:WCbackground}

Wonderful compactifications were introduced by De Concini and Procesi \cite{DCP} in the context of linear subvarieties of a projective space.  Roughly speaking, given a collection of linear subvarieties in $\P^n$, a wonderful compactification is a way of replacing $\P^n$ by a different ambient variety in such a way that the complement of the linear subvarieties is preserved but the subvarieties themselves are replaced by a divisor with normal crossings.  The construction of wonderful compactifications was later generalized by Li Li \cite{LiLi} to more general collections of subvarieties in a smooth variety. In this section, we briefly review the necessary definitions for the current work, but we refer the reader to many more in-depth references---including \cite{DCP, Denham, Feichtner, FY, LiLi}---for details.  Throughout, we consider all varieties over $\C$.

\subsection{Wonderful compactifications of arrangements of subvarieties}
\label{subsec:WCbackground}

Let $Y$ be a smooth variety.  An {\bf arrangement} of subvarieties of $Y$ is a finite collection of smooth subvarieties and that pairwise intersect ``cleanly'' (see \cite[Definition 2.1]{LiLi}).  If
\[\calA = \{X_1, \ldots, X_r\}\]
is an arrangement, we denote by $\calL_{\calA}$ the intersection lattice of $\calA$; this is the poset of all intersections of subsets of $\calA$, ordered by reverse inclusion.  In particular, the unique minimal element of $\calL_{\calA}$ is $\hat{0} = Y$, which we view as the empty intersection, and the unique maximal element is $\hat{1} = \emptyset$. By the {\bf complement} of $\calA$, we mean
\[Y^{\circ}:= Y \setminus \bigcup_{i=1}^r X_i.\]

Some of the subvarieties in $\calA$ may intersect non-transversally, and the goal of a wonderful compactification of $Y^{\circ}$ is to modify the ambient variety $Y$ in such a way that the arrangement is replaced by a simple normal crossings divisor.  It is not surprising that the way to do so is to perform an iterated blow-up.  While one can obtain a wonderful compactification by blowing up at every element of $\calL_{\calA}$ (in a carefully-prescribed order explained below), some subsets of $\calA$ may already intersect transversally, so one can often obtain a compactification with similar properties by blowing up only at a subset of $\calL_{\calA}$.  The particular subsets that give rise to wonderful compactifications are known as {\bf building sets}; for the precise definition, see \cite[Definition 2.2]{LiLi}.  The most important example of a building set for the current work is the {\bf maximal building set} $\calG := \calL_{\calA} \setminus \{\widehat{0}\}$, which corresponds to blowing up every intersection of elements of $\calA$.

In general, a choice of a building set $\calG \subseteq \calL_{\calA} \setminus \{\widehat{0}\}$ gives rise to a {\bf wonderful compactification} $\overline{Y}_{\calG}$ of $Y^{\circ}$ in the following way.  First, choose an ordering of the elements of $\calG$ that is compatible with inclusion; that is, let
\[\calG = \{G_1, \ldots, G_N\}\]
in which $i \leq j$ if $G_i \subseteq G_j$.  Then, perform the following sequence of blow-ups:
\begin{itemize}
    \item blow up $Y$ along $G_1$,
    \item blow up the result along the proper transform of $G_2$,
    \item blow up the result along the proper transform of $G_3$,
\end{itemize}
and so on.  Then, as shown in \cite[Proposition 2.13]{LiLi}, the wonderful compactificaiton $\overline{Y}_{\calG}$ is the end result after blowing up along the proper transform of $G_{N}$.   

Since the blow-ups that form $\overline{Y}_{\calG}$ are only at intersections of the subvarieties $X_i$, there is an inclusion
\[Y^{\circ} \hookrightarrow \overline{Y}_{\calG},\]
and we refer to the complement $\overline{Y}_{\calG} \setminus Y^{\circ}$ as the {\bf boundary} of the wonderful compactification.  Among the ``wonderful'' properties of $\overline{Y}_{\calG}$ is the extent to which the structure of this boundary is encoded in the combinatorics of $\calG$.  In particular, the boundary is a union of divisors $D_G$ for each nonempty $G \in \calG$, and the intersection $D_{T_1} \cap \cdots \cap D_{T_r}$ is nonempty if and only if $\{T_1, \ldots, T_r\}$ forms a {\bf $\calG$-nested set}.  The definition of $\calG$-nested set is purely combinatorial and can be stated in a number of equivalent ways (see, for example, \cite[Definition 2.3]{LiLi} or \cite[Definition 3.2]{Feichtner}).  In the case where $\calG$ is the maximal building set, a $\calG$-nested set is precisely a chain in $\calL_{\calA} \setminus \{\widehat{0}\}$ as a poset.

\subsection{Wonderful compactifications of hyperplane arrangements}
\label{subsec:HAbackground}

In their original work introducing wonderful compactifications \cite{DCP}, De Concini and Procesi proved that if $\calA$ is an arrangement of hyperplanes in projective space, then the cohomology (which is isomorphic to the Chow ring, for example by \cite{keel1992intersection}) of a wonderful compactification can be read off combinatorially from the lattice $\calL_\calA$ and its building set.  Feichtner and Yuzvinsky reinterpreted this calculation in \cite{FY}, constructing a fan $\Sigma_{\calG}$ associated to any lattice $\calL$ with building set $\calG$ and proving that, in the case where $\calL$ is the intersection lattice of a hyperplane arrangement in projective space, the Chow ring of the toric variety $X_{\Sigma_{\calG}}$ coincides with De Concini--Procesi's calculation of the Chow ring of the wonderful compactification $\overline{Y}_{\calG}$ of the complement of $\calA$.  In this section, we review the parts of this story that are necessary for what follows.

Let $\calA = \{H^0, \ldots, H^{r-1}\}$ be a collection of hyperplanes in $\P^k$.  We assume in what follows that $\calA$ is {\bf essential}, meaning that
\[\bigcap_{i=0}^{r-1} H^i = \emptyset.\]
In this case, there is an inclusion
\[i:\P^k \hookrightarrow \P^{r-1}\]
under which $H^0, \ldots, H^{r-1}$ map to the coordinate hyperplanes; namely, if $H^i = \calV(f_i)$ for linear polynomials $f_i \in \C[x_0, \ldots, x_k]$, then
\[i(p) = [f_0(p): \cdots : f_{r-1}(p)].\]
It follows that $i$ maps the complement
\[Y^{\circ}= \P^k \setminus \bigcup_{i=0}^{r-1} H^i\]
of $\calA$ into the complement of the coordinate hyperplanes in $\P^{r-1}$, or in other words into the algebraic torus
\[\mathbb{T}^{r-1} = (\C^*)^{r-1}.\]
By identifying $Y^{\circ}$ with its image under $i$, then, we can view $Y^{\circ}$ as a {\bf very affine variety}---that is, a closed subvariety of a torus.

For any building set $\calG \subseteq \calL_{\calA} \setminus \{\hat{0}\}$, one defines the {\bf nested set fan} $\Sigma_{\calG}$ of $(\calL_\calA, \calG)$ as follows.  First, let
\[V_{\calA} := \R^r/\R,\]
where the quotient is by the diagonal, and denote the images of the standard basis vectors by $e^0, \ldots, e^{r-1}$.  For each $G \in \calG$, define
\[v_G := \sum_{H^j \supseteq G} e^j \in V_{\calA}.\]
Then $\Sigma_{\calG}$ is defined as the fan in $V_{\calA}$ whose cones are
\[\sigma_S:= \text{Cone}\{v_G \; | \; G \in S\} \subseteq V_{\calA}\]
for each $\calG$-nested set $S \subseteq \calG$.

Note that the toric variety $X_{\Sigma_{\calG}}$ has $\mathbb{T}^{r-1}$ as its torus, so in particular, we have
\[Y^{\circ} \subseteq \mathbb{T}^{r-1} \subseteq X_{\Sigma_{\calG}}.\]
By reinterpreting $\Sigma_{\calG}$ in terms of a stellar subdivision procedure as in \cite[Section 6]{FY} (which corresponds to regarding $X_{\Sigma_{\calG}}$ as an iterated blow-up of $\P^{r-1}$), one sees that the wonderful compactification $\overline{Y}_{\calG}$ is equal to the closure of $Y^{\circ}$ inside of $X_{\Sigma_{\calG}}$.  Moreover, by \cite[Corollary 2]{FY}, the inclusion 
\[\overline{Y}_{\calG} \hookrightarrow X_{\Sigma_{\calG}}\]
is a Chow equivalence.  This allows one to give a presentation of $A^*(\overline{Y}_{\calG})$ that can be read off directly from the combinatorics of the lattice $\calL_{\calA}$ with its building set $\calG$.

\begin{remark}
The moduli space $\M_{0,n}$ can be obtained as the wonderful compactification of the braid arrangement $\calA_{n-2}$ (the arrangement of hyperplanes $\{x_i = x_j\} \subseteq \P^{n-3}$ for all $i \neq j$), with an appropriate choice of building set \cite[Section 4.3]{DCP}.  In this case, the above results lead to an elegant presentation of the Chow ring of $\M_{0,n}$, as described in \cite[Section 4.2]{Feichtner}.  Moreover, the nested set fan can be interpreted in this context as the Bergman fan of a particular matroid, or as the moduli space of tropical curves.  These results were generalized in \cite{CHMR} to all genus-zero Hassett spaces with weight system of ``heavy/light'' type, leading to a presentation of the Chow ring of such spaces in \cite{KKL}.
\end{remark}

\subsection{Wonderful compactifications of product arrangements}
\label{subsec:wcproduct}

The case of interest in the current work is the moduli space $\L^r_n$, which, as we prove below, is a wonderful compactification of the complement of an arrangement of hyperplanes not in a single projective space but in a product of projective spaces.  Although such wonderful compactifications have been constructed via iterated blow-up (through the much more general work of Li Li described above), there is not, to our knowledge, a construction in this setting as the closure inside of a toric variety analogous to $X_{\Sigma_{\calG}}$.  We prove such a presentation in this subsection, and as a result, we obtain an identification of the Chow ring of such wonderful compactifications with the Chow ring of a toric variety that can be read off combinatorially from the intersection lattice and its building set.

Here, and in what follows, for positive integers $n$ and $r$ we use the notation
\[[n] := \{1, 2, \ldots, n\}\]
and
\[\Z_r := \{0,1,2, \ldots, r-1\}.\]
We choose these sets to index the hyperplanes in a product arrangement for consistency with the application to $\L^r_n$ that follows.

For each $i \in [n]$, fix positive integers $r_i$ and $k_i$ and an essential hyperplane arrangement \begin{equation}
\label{eq:Ai}
\mathcal{A}_i = \{H^0_i, \ldots, H^{r_i-1}_i\}.
\end{equation} inside $\P^{k_i}$. Let $Y_i^{\circ} \subseteq \P^{k_i}$ denote the complement of the arrangement $\calA_i$.  Then the product
\[Y^{\circ}:= Y_1^{\circ} \times \cdots \times Y_n^{\circ} \subseteq \P^{k_1} \times \cdots \times \P^{k_n}\]
is also the complement of a hypersurface arrangement: namely, it is the complement of
\[\mathcal{A}:= \{\widetilde{H}^j_i \; | \; i \in [n], \; j \in \mathbb{Z}_{r_i}\},\]
in which
\[\widetilde{H}_i^j:= p_i^{-1}\left(H_i^j\right)\]
is the pullback of $H_i^j \subseteq \P^{k_i}$ under the projection $p_i: \P^{k_1} \times \cdots \times \P^{k_n} \rightarrow \P^{k_i}$ to the $i$th factor.  We refer to $\calA$ as the {\bf product arrangement} induced by $\calA_1, \ldots, \calA_n$.

\begin{remark}
\label{rem:linear}
The variety $Y^{\circ}$ is very affine, since the embeddings $Y_i^{\circ} \hookrightarrow \mathbb{T}^{r_i-1}$ described in Section~\ref{subsec:HAbackground} combine to give
\begin{equation}
    \label{eq:Ycirc}
Y^{\circ} \hookrightarrow \mathbb{T}^{r_1-1} \times \cdots \times \mathbb{T}^{r_n-1} = \mathbb{T}^r,
\end{equation}
where $r := r_1 + \cdots + r_n - n$.  Moreover, $Y^{\circ}$ is linear in the sense of \cite{Gross} (that is, it is cut out by linear equations in coordinates on $\mathbb{T}^r$), because each factor $Y_i^{\circ} \hookrightarrow \mathbb{T}^{r_i-1}$ is linear.  This observation plays a key role in the proof of Theorem~\ref{thm:wcproduct} below.

In fact, for Theorem~\ref{thm:wcproduct}, it is enough to know that $Y^{\circ}$ is quasilinear in the sense of \cite{Schock}.  Schock introduced quasilinear varieties in \cite{Schock} as a generalization of linear varieties that retains the key property that, if $Y^{\circ} \hookrightarrow \mathbb{T}$ is quasilinear and $\overline{Y} \hookrightarrow X_{\Sigma}$ is a ``tropical compactification" of $Y^{\circ}$, then $\overline{Y}$ is Chow-equivalent to $X_{\Sigma}$.  Given that \cite[Theorem 6.4]{Schock} shows that products of quasilinear varieties are quasilinear, it is immediate from \eqref{eq:Ycirc} that $Y^{\circ}$ is quasilinear in our case.
\end{remark}

\begin{example}
\label{example:1}
A simple but illustrative example, which is relevant for the application to $\L^r_n$ below, is to take $n=2$ and set
\[\calA_1 = \calA_2 := \{[1:1], \; [1:-1]\} \subseteq \P^1.\]
Then the product arrangement $\calA$ consists of four hyperplanes in $\P^1 \times \P^1$:
\begin{align}
    \label{eq:calA}
\calA &= \{ \widetilde{H}_1^0, \;\widetilde{H}_1^1, \; \widetilde{H}_2^0, \; \widetilde{H}_2^1\}\\
&= \nonumber \Big\{ \{[1:1]\} \times \P^1, \; \{[1:-1]\} \times \P^1, \; \P^1 \times \{[1:1]\}, \; \P^1 \times \{[1:-1]\} \Big\} \subseteq \P^1 \times \P^1.
\end{align}
In this case, the embeddings $i_1: Y_1^{\circ} \hookrightarrow \mathbb{T}^1$ and $i_2: Y_2^{\circ} \hookrightarrow \mathbb{T}^1$ are equal and are in fact isomorphisms; indeed, they both come from the embedding (in fact, change of coordinates) $i_1 = i_2: \P^1 \rightarrow \P^1$ given by
\[[x:y] \mapsto [x-y:x+y],\]
which sends the hyperplanes in $\calA_1 = \calA_2$ to the coordinate hyperplanes in $\P^1$.  Thus, the product
\[i = i_1 \times i_2: \P^1 \times \P^1 \rightarrow \mathbb{T}^1 \times \mathbb{T}^1\]
sends $Y^{\circ}$ isomorphically to $\mathbb{T}^1 \times \mathbb{T}^1 = \mathbb{T}^2$.
\end{example}

The lattice $\calL_\calA \setminus \{\hat{1}\}$ is the product of the lattices $\calL_{\calA_i} \setminus \{\hat{1}_i\}$ with the product order, where $\hat{1}_i$ denotes the maximal element $\emptyset$ in the intersection lattice of the arrangement $\calA_i$.  From this one finds two combinatorial consequences that are important in what follows.

\begin{lemma} \label{lem:productprop}
Fix building sets $\calG_1, \ldots, \calG_n$ for the arrangements $\calA_1, \ldots, \calA_n$, respectively, and assume that $\hat{1}_i \in \calG_i$ for at least one $i$.  For each $i$, view $\calG_i$ as a subset of $\calL_{\calA}$ by identifying $X \in \calG_i$ with $p_i^{-1}(X) \in \calL_{\calA}$.  Then we have the following:
\begin{enumerate} 
    \item[(a)] The union $\bigcup_{j=1}^n\calG_j$ is a building set for $\calL_\calA$.
    \item[(b)] If $S_i \subseteq \calG_i$ for each $i$, then
    \[S_i \text{ is } \calG_i\text{-nested for each }i \; \Leftrightarrow \; \bigcup_{j=1}^n S_j \text{ is } \left(\bigcup_{j=1}^n \calG_j\right)\text{-nested}.\]
\end{enumerate}
\end{lemma}
\proof \textit{(a)} By the definition of building sets (see, for example, \cite[Definition 1]{FY}), we must prove that for any $X \in \calL_{\calA}$, there is an isomorphism of posets
\begin{equation}
    \label{eq:buildingsetL}
[\hat{0}, X] \cong \prod_{Z \in \text{max}\big((\calG_1 \cup \cdots \cup \calG_n) \cap [\hat{0}, X]\big)} [\hat{0}, Z].
\end{equation}
If $X = \hat{1}$, then the condition that $\hat{1}_i \in \calG_i$ ensures that both sides of \eqref{eq:buildingsetL} are the full lattice $\calL_{\calA}$.  Suppose, then, that $X \neq \hat{1}$.  In this case under the isomorphism of $\calL_\calA \setminus \{\hat{1}\}$ with the product of the lattices $\calL_{\calA_i} \setminus \{\hat{1}\}$, we have $X= \prod_{i=1}^n X_i$ for $X_i \in \calL_{\calA_i}$.  Thus,
\[[\hat{0}, X] \cong \left[\hat{0}, \; \prod_{i=1}^n X_i\right] \cong \prod_{i=1}^n[\hat{0}, X_i] \cong \prod_{i=1}^n \prod_{Z_i \in \text{max}(\calG_i \cap [\hat{0}, X_i])} [\hat{0}, Z_i],\]
where the last isomorphism follows from the fact that each $\calG_i$ is a building set.  It is straightforward to check that this is equivalent to \eqref{eq:buildingsetL}.

\textit{(b)} We denote
\[S:= \bigcup_{j=1}^n S_j,\]
and we use the characterization of nested sets given in \cite[Section 2.4, Lemma (1)]{DCP}: a subset $T$ of a building set $\calH$ is $\calH$-nested if, given pairwise incomparable elements $X_1, \ldots, X_t \in T$ in which $t \geq 2$, the join $X_1 \vee \cdots \vee X_t$ is not in $\calH$.

Suppose that each $S_i$ is $\calG_i$-nested.  To see that $S$ is $\left(\bigcup_{j=1}^n\calG_j\right)$-nested, let $X_1, \ldots, X_t \in S$ be pairwise incomparable elements with $t \geq 2$.  (If no such elements exist, then $S$ is automatically nested.)  If at least two of these elements belong to different factors $S_i$, then their join is not in $\bigcup_{j=1}^n \calG_j$, so we are done.  Thus, all that remains is the possibility that $X_1, \ldots, X_t \in S_i$ for some $i$, in which case the fact that $S_i$ is $\calG_i$-nested implies that
\[X_1 \vee \cdots \vee X_t \notin \calG_i\]
and hence this join is not in $\bigcup_{j=1}^n \calG_j$.

Conversely, suppose that $S$ is $\left(\bigcup_{j=1}^n\calG_j\right)$-nested.  To see that $S_i$ is $\calG_i$-nested for each $i$, let $X_1, \ldots, X_t \in S_i$ be pairwise incomparable elements with $t \geq 2$.  Since $S$ is $\left(\bigcup_{j=1}^n\calG_j\right)$-nested, we have
\[X_1 \vee \cdots \vee X_n \notin \bigcup_{j=1}^n \calG_j,\]
so in particular, this join is not in $\calG_i$.  
\endproof

We are now prepared to define ``nested set fans'' in the product setting by direct analogy to the situation described in Section~\ref{subsec:HAbackground}. 
 
\begin{definition}
\label{def:nestedsetfan}
Let $\calA_1, \ldots, \calA_n$ be hyperplane arrangements as in \eqref{eq:Ai}, let $\calA$ be the induced product arrangement, and let $V_{\calA}$ be the vector space
\[V_\calA := \R^{r_1}/\R \times \cdots \times \R^{r_n}/\R,\]
where each quotient is by the diagonal and we denote the images of the standard basis vectors in the $i$th factor by $e_i^0, \ldots, e_i^{r_i-1}$.  For any $G \in \calL_{\calA} \setminus \{\hat{0}\}$, define
\[v_G := \sum_{\widetilde{H}_i^j \supseteq G} e_i^j \in V_{\calA}.\]
Then, given any building set $\calG \subseteq \calL_{\calA} \setminus \{\hat{0}\}$, the {\bf nested set fan} for $(\calL_\calA, \calG)$ is the fan $\Sigma_{\calG}$ in $V_{\calA}$ whose cones are
\begin{equation}
    \label{eq:sigmaS}
\sigma_S := \text{Cone}\{ v_G \; | \; G \in S\} \subseteq V_{\calA}
\end{equation}
for each $\calG$-nested set $S \subseteq \calG$.
\end{definition}

\begin{example}
\label{example:2}
In the case of Example~\ref{example:1}, one has $n=2$ and $r_1 = r_2 = 2$, so
\[V_{\calA} = \R^2/\R \times \R^2/\R \cong \R^2.\]
Let $\calG$ be the maximal building set, so that $\calG$-nested sets are precisely chains in $\calL_\calA \setminus \{\hat{0}\}$ as a poset---in other words, nested collections of intersections of the sets $\widetilde{H}^i_j$ listed in equation~\eqref{eq:calA}.  The nested set fan $\Sigma_{\calG}$ in this example is depicted in Figure~\ref{fig:Sigmarn}.  In particular, the shaded cone is
\[\text{Cone}(e_2^0, \;e_1^1 + e_2^0),\]
which is the cone $\sigma_S$ for the $\calG$-nested set $S = \{\widetilde{H}_2^0, \; \widetilde{H}_1^1 \cap \widetilde{H}_2^0\}$.
\end{example}

The only difference between Definition~\ref{def:nestedsetfan} and Feichtner--Yuzvinksy's nested set fan described in Section~\ref{subsec:HAbackground} is the quotients by $\R$ in $V_{\calA}$ corresponding to each projective space factor.  The point, however, is that these quotients do not affect the key step in Feichtner--Yuzvinsky's argument that $X_{\Sigma_{\calG}}$ is Chow-equivalent to the wonderful compactification $\overline{Y}_{\calG}$, which is a re-expression of $\Sigma_{\calG}$ in terms of a stellar subdivision procedure; see \cite[Theorem 4]{FY} and Lemma~\ref{lem:stellar} below.

In particular, we have the following analogue for product arrangements of the known results for hyperplane arrangements in projective space.

\begin{theorem}
\label{thm:wcproduct}
Let $\calA_1, \ldots, \calA_n$ be essential hyperplane arrangements in respective projective spaces $\P^{k_1}, \ldots, \P^{k_n}$, let $\calA$ be the induced product arrangement in $\P^{k_1} \times \cdots \times \P^{k_n}$, and let $Y^{\circ} \subseteq \P^{k_1} \times \cdots \P^{k_n}$ be the complement of $\calA$.  Let $\calG$ be any building set for $\calL_\calA$, and let $\Sigma_{\calG}$ be the nested set fan for $(\calL_\calA, \calG)$. 
Then there is an embedding
\[Y^{\circ} \hookrightarrow X_{\Sigma_{\calG}}\]
such that the wonderful compactification $\overline{Y}_{\calG}$ is the closure of $Y^{\circ}$ in $X_{\Sigma_{\calG}}$.  Moreover, the inclusion of $\overline{Y}_{\calG}$ into $X_{\Sigma_{\calG}}$ is a Chow equivalence:
\[A^*(\overline{Y}_{\calG}) = A^*(X_{\Sigma_{\calG}}).\]
\end{theorem}

In order to prove this theorem, we first observe that a building set $\calG$ for $\calL_\calA$ induces building sets $\calG_1, \ldots, \calG_n$ for $\calL_{\calA_1}, \ldots, \calL_{\calA_n}$, respectively:
\begin{equation*}
\calG_i := \left\{X \in \calL_{\calA_i} \; \left| \; p_i^{-1}(X) \in \calG\right.\right\}.
\end{equation*}
Thus, one can define a nested set fan $\Sigma_{\calG_i}$ for each $i$, which is a fan in $\R^{r_i}/\R$.  While $\Sigma_\calG$ is not equal to the product $\Sigma_{\calG_1} \times \cdots \times \Sigma_{\calG_n}$, it is equal to a stellar subdivision of that product, as the following lemma verifies.  

 \begin{lemma}
 \label{lem:stellar}
 Let $\calA$ be a product arrangement induced by arrangements $\calA_1, \ldots, \calA_n$, let $\calG$ be a building set for $\calL_\calA$, and let $\calG_1, \ldots, \calG_n$ be the induced building sets for $\calL_{\calA_1}, \ldots, \calL_{\calA_n}$.  Viewing each $\calG_i$ as a subset of $\calG$ by identifying $X \in \calG_i$ with $p_i^{-1}(X) \in \calG$, write
\[\calG \setminus \bigcup_{i=1}^n \calG_i = \{C_1, \ldots, C_N\},\]
 where the elements are ordered in such a way that $i \leq j$ whenever $C_i \subseteq C_j$.  Then $\Sigma_{\calG}$ is obtained from $\Sigma_{\calG_1} \times \cdots \times \Sigma_{\calG_n}$ by stellar subdivision at the vector $v_{C_1}$, then the vector $v_{C_2}$, and so on.
 \end{lemma}

\begin{proof}
It suffices to assume that $\hat{1} \in \calG$ (and therefore $\hat{1}_i \in \calG_i$ for each $i$), because if $G = \hat{1}$ then $v_G = 0 \in V_{\calA}$, so including $\hat{1}$ in $\calG$ does not affect the nested set fan.  Thus, in view of Lemma \ref{lem:productprop}(a), we see that $\bigcup_{i=1}^n\calG_i$ is a building set for $\calL_\calA$.  It therefore induces a nested set fan, and we claim that
\begin{equation}
    \label{eq:productnestedsetfan}
\Sigma_{\calG_1} \times \cdots \times \Sigma_{\calG_n} = \Sigma_{\calG_1 \cup \cdots \cup \calG_n}.
\end{equation}
Indeed, the cones of $\Sigma_{\calG_1 \cup \cdots \cup \calG_n}$ are, by definition, of the form $\sigma_S$ for each $(\bigcup_{i=1}^n \calG_i)$-nested set $S$.  By Lemma \ref{lem:productprop}(b), these are precisely the cones
\[\sigma_{S_1 \cup \cdots \cup S_n} = \sigma_{S_1} \times \cdots \times \sigma_{S_n}\]
in which $S_i \subseteq \calG_i$ is $\calG_i$-nested for each $i$, which are the cones of $\Sigma_{\calG_1} \times \cdots \times \Sigma_{\calG_n}$.

On the other hand, by \cite[Theorem 4.2]{FM}, the inclusion of building sets $\left(\bigcup_{i=1}^n\calG_i\right) \subseteq \calG$ implies that $\Sigma_{\calG}$ is obtained from $\Sigma_{\calG_1 \cup \cdots \cup \calG_n}$ by the sequence of stellar subdivision as claimed.  Thus, by \eqref{eq:productnestedsetfan}, the proof is complete.
\end{proof}

\begin{example}
As an illustration of Lemma~\ref{lem:stellar}, let $\calA$ again be the product arrangement of Examples~\ref{example:1} and \ref{example:2}, and let $\calG$ be its maximal building set.  Explicitly, $\calG$ consists of the four hyperplanes $\widetilde{H}_i^j$ listed in \eqref{eq:calA} as well as the intersections $\widetilde{H}_1^j \cap \widetilde{H}_2^k$ for all $j,k \in \{0,1\}$, whereas
\begin{align*}
\calG_1 &= \{H_1^0, \; H_1^1\} = \{[1:1], [1:-1]\},\\
\calG_2 &= \{H_2^0, \; H_2^1\} = \{[1:1], [1:-1]\}.
\end{align*}
One has
\[V_{\calA_1} = V_{\calA_2} = \R^2/\R \cong \R,\]
and $\Sigma_{\calG_1} = \Sigma_{\calG_2}$ is the fan in this vector space consisting of two rays pointing in opposite directions together with the origin. Explicitly, the positive-dimensional cones in $\Sigma_{\calG_1}$ are 
\[
\Big\{\text{Cone}(e_1^0), \; \text{Cone}(e_1^1)\Big\},
\]
and the positive-dimensional cones in $\Sigma_{\calG_2}$ are 
\[
\Big\{\text{Cone}(e_2^0), \; \text{Cone}(e_2^1)\Big\},
\]
from which one sees that the product $\Sigma_{\calG_1} \times \Sigma_{\calG_2}$ has four two-dimensional cones
\[\text{Cone}(e_1^0, e_2^0), \; \text{Cone}(e_1^0, e_2^1), \; \text{Cone}(e_1^1, e_2^0), \; \text{Cone}(e_1^1, e_2^1).\]
The fan $\Sigma_{\calG}$, which we considered in Example~\ref{example:2}, is obtained from this product by stellar subdivision along the four vectors $e_1^j + e_2^k$ corresponding to the four elements $\widetilde{H}_1^j \cap \widetilde{H}_2^k$ of $\calG \setminus (\calG_1 \cup \calG_2)$.  See Figure~\ref{fig:stellar} for an illustration, though note that the fan $\Sigma_{\calG_1} = \Sigma_{\calG_2}$ is denoted by $\Sigma_2$ in that figure, and the fan $\Sigma_{\calG}$ is denoted by $\Sigma_2^2$, for consistency with the general notation for $\L^r_n$ established below.
\end{example}

The key upshot of Lemma~\ref{lem:stellar} is the following.  By \cite[Theorem 4]{FY}, each of the fans $\Sigma_{\calG_i}$ can be obtained from the fan for $\P^{r_i-1}$ by a two-step process: first, one performs successive stellar subdivision along the vectors $v_Z$ for $Z \in\calG_i$, which produces a fan in which all cones have the form $\sigma_S$ for $S \subseteq \calG_i$, and second, one removes the open cones $\sigma_S$ for which $S$ is not $\calG_i$-nested.  Thus, Lemma~\ref{lem:stellar} says that $\Sigma_{\calG}$ can similarly be obtained from the fan for $\P^{r_1-1} \times \cdots \times \P^{r_n-1}$ by first performing successive stellar subdivisions along the vectors $v_G$ for all $G \in \calG$, and then removing the open cones $\sigma_S$ for which $S$ is not $\left(\bigcup_{i=1}^n \calG_i\right)$-nested.  

Equipped with these observations, we are ready for the proof of Theorem~\ref{thm:wcproduct}.

\begin{proof}[Proof of Theorem~\ref{thm:wcproduct}]
The fact that there is an embedding $Y^{\circ} \hookrightarrow X_{\Sigma_{\calG}}$ is immediate: by Remark~\ref{rem:linear}, we have an embedding of $Y^{\circ}$ into the torus $\mathbb{T}^{r_1-1} \times \cdots \times \mathbb{T}^{r_n-1}$, which is the torus for the toric variety $\Sigma_{\calG}$.

To see that the closure of $Y^{\circ}$ in $X_{\Sigma_{\calG}}$ is indeed $\overline{Y}_{\calG}$, write
\[\calG = \{W_1, \ldots, W_M\},\]
again ordered in such a way that $i \leq j$ whenever $W_i \subseteq W_j$.  Then Li Li's construction of wonderful compactifications in \cite[Definition 2.12]{LiLi} shows that $\overline{Y}_{\calG}$ is an iterated blow-up of $\P^{k_1} \times \cdots \times \P^{k_n}$ along $W_1, \ldots, W_M$.  Now, let
\[i: \P^{k_1} \times \cdots \times \P^{k_n} \hookrightarrow \P^{r_1-1} \times \cdots \times \P^{r_n-1}\]
be the product of the embeddings described in Section~\ref{subsec:HAbackground}, under which the elements of $\calA$ are mapped to torus-invariant strata.  In particular, let $Z_1, \ldots, Z_M$ be torus-invariant strata such that $i^{-1}(Z_j) = W_j$ for each $j$.  Then, by the blow-up closure lemma (see \cite[Lemma 22.2.6]{VakilBook}), one can view $\overline{Y}_{\calG}$ as the closure of the image of
\[\P^{k_1} \times \cdots \times \P^{k_n} \setminus \bigcup_{i=1}^M W_j\]
in the iterated blow-up of $\P^{r_1-1} \times \cdots \times \P^{r_n-1}$ along $Z_1, \ldots, Z_M$.  This is the same as the closure of the image of $Y^{\circ}$ in this iterated blow-up, since replacing the above complement by $Y^{\circ}$ only adds points that avoid $Z_1, \ldots, Z_M$.

The iterated blow-up of $\P^{r_1-1} \times \cdots \times \P^{r_n-1}$ along $Z_1, \ldots, Z_M$ is a toric variety whose fan has cones of the form $\sigma_S$ for $S \subseteq \calG$, and, by the discussion immediately following the proof of Lemma~\ref{lem:stellar} above, one can obtain $X_{\Sigma_{\calG}}$ from this toric variety by removing all of the open strata corresponding to cones $\sigma_S$ in which $S$ is not $\left(\bigcup_{i=1}^n \calG_i\right)$-nested.  Since $\bigcup_{i=1}^n \calG_i \subseteq \calG$, such sets are also not $\calG$-nested.  It follows that removing these cones does not affect the closure of $Y^{\circ}$, because the fact that the boundary strata of $\overline{Y}_{\calG}$ are indexed by $\calG$-nested sets (see \cite[Section 3.2]{DCP}) means that it avoids the blow-ups corresponding to non-nested sets.  Thus, $\overline{Y}_{\calG}$ is indeed the closure of $Y^{\circ}$ in $X_{\Sigma_{\calG}}$.

Finally, to see that the inclusion $\overline{Y}_{\calG} \hookrightarrow X_{\Sigma_{\calG}}$ is a Chow equivalence, we recall from Remark~\ref{rem:linear} that $Y^{\circ} \subseteq \mathbb{T}^r$ is a linear variety, which implies by \cite[Theorem 1.1]{Gross} that such a Chow equivalence holds so long as $\overline{Y}_{\calG} \subseteq X_{\Sigma_{\calG}}$ is a tropical compactification, meaning that $|\Sigma_{\calG}| = \text{Trop}(Y^{\circ})$ and the multiplication map $\mathbb{T}^r \times \overline{Y}_{\calG} \rightarrow X_{\Sigma_{\calG}}$ is faithfully flat.  This is indeed the case: each $\overline{Y}_{\calG_i}$ is a tropical compactification and, by Lemma~\ref{lem:stellar}, there is a proper toric morphism
 \[\Sigma_{\calG} \rightarrow \Sigma_{\calG_1} \times \cdots \times \Sigma_{\calG_n},\]
so the fact that $\overline{Y}_{\calG}$ is a tropical compactification follows from \cite[Proposition 2.5]{Tevelev}.
\end{proof}

\section{The moduli space of curves with cyclic action}
\label{sec:Lrnbackground}

In this section, we review the definition and necessary properties of the moduli space $\L^r_n$ introduced in \cite{CDHLR}, and we prove that it is a wonderful compactification of a product arrangement in $(\P^1)^n$.  Throughout, we assume that $r \geq 2$.

\subsection{Background on \texorpdfstring{$\Lrn$}{the moduli space}}

The objects parameterized by $\L^r_n$ are {\bf stable $(r,n)$-curves}.  The underlying curve $C$ in such an object is an ``$r$-pinwheel curve,'' which is a rational curve consisting of a central projective line from which $r$ equal-length chains of projective lines (``spokes'') emanate.  This curve is equipped with an order-$r$ automorphism $\sigma$, as well as marked points as follows:
\begin{itemize}
    \item two distinct fixed points $x^+$ and $x^-$ of $\sigma$;
    \item $n$ labeled $r$-tuples $(z_1^{0}, \ldots, z_1^{r-1}), \ldots, (z_n^{0}, \ldots, z_n^{r-1})$ of points $z_i^j \in C$ satisfying
    \[\sigma(z_i^j) = z_i^{j+1 \!\!\!\!\mod r}\]
    for each $i$ and $j$, where we allow $z_i^j = z_{i'}^{j'}$ and $z_i^j = x^{\pm}$;
    \item an additional labeled $r$-tuple $(y^{0}, \ldots, y^{r-1})$ satisfying
    \[\sigma(y^\ell) = y^{\ell+1 \!\!\!\!\mod r}\]
    for each $\ell$, whose elements are distinct from one another as well as from $x^{\pm}$ and $z_i^j$.
\end{itemize}
These marked points are subject to a stability condition, the details of which can be found in \cite[Section 2.1]{CDHLR}.  We refer to each tuple $(z_i^0, \ldots, z_i^{r-1})$ as a ``light orbit'' of $\sigma$ and the tuple $(y^0, \ldots, y^{r-1})$ as the ``heavy orbit.''  See Figure~\ref{fig:samplecurve} for an example of a stable $(r,n)$-curve.  

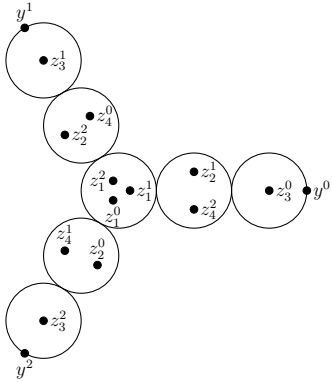
\begin{figure}[ht]
\begin{center}
    \begin{tikzpicture}[scale=1,  every node/.style={scale=0.6}]
    \filldraw (0.15,0) circle (0.05cm) node[right]{$z^1_1$};
    \filldraw (-0.075,0.1299) circle (0.05cm) node[left]{$z^2_1$};
    \filldraw (-0.075,-0.1299) circle (0.05cm) node[below]{$z^0_1$};
    \filldraw(1,0.25) circle (0.05cm) node[right]{$z^1_2$};
    \filldraw(-0.7165, 0.741) circle (0.05cm) node[right]{$z^2_2$};
    \filldraw(-0.283, -0.991) circle (0.05cm) node[above]{$z^0_2$};
    \filldraw(1,-0.25) circle (0.05cm) node[right]{$z^2_4$};
    \filldraw(-0.383,0.991) circle (0.05cm) node[right]{$z^0_4$};
    \filldraw(-0.717,-0.8) circle (0.05cm) node[above]{$z^1_4$};
    \filldraw (2,0) circle (0.05cm) node[right]{$z^0_3$};
    \filldraw (-1,1.732) circle (0.05cm) node[right]{$z^1_3$};
    \filldraw (-1,-1.732) circle (0.05cm) node[right]{$z^2_3$};
    \draw (0,0) circle (0.5cm);
    \draw (1,0) circle (0.5cm);
    \draw (2,0) circle (0.5cm);
    \filldraw (2.5,0) circle (0.05cm) node[right]{$y^0$};
    \draw (-0.5,0.866) circle (0.5cm);
    \draw (-1,1.732) circle (0.5cm);
    \filldraw (-1.25,2.165) circle (0.05cm) node[above]{$y^1$};
    \draw (-0.5,-0.866) circle (0.5cm);
    \draw (-1,-1.732) circle (0.5cm);
    \filldraw (-1.25,-2.165) circle (0.05cm) node[below]{$y^2$};
        \end{tikzpicture}
\end{center}
\caption{A stable $(3,4)$-curve, where each circle represents a $\P^1$ component and $\sigma$ is the rotational automorphism.  Not pictured are the marked points $x^+$ and $x^-$, which are the two fixed points of $\sigma$ and must both lie on the central component.}
\label{fig:samplecurve}
\end{figure}

In \cite[Theorem 3.5]{CDHLR}, a fine moduli space $\L^r_n$ for stable $(r,n)$-curves is constructed, whose $B$-points correspond to families of stable $(r,n)$-curves over the base scheme $B$ as defined in \cite[Definition 2.5]{CDHLR}.  More precisely, there is a connected component $\L^r_n(\zeta)$ for any choice of primitive $r$th root of unity $\zeta$, all of which are isomorphic to one another, and the moduli space $\L^r_n$ is the disjoint union of these connected components.  In what follows, we will assume that $\zeta$ is fixed and we will therefore abuse notation by referring to the space $\L^r_n$ when we in fact mean a single component $\L^r_n(\zeta)$.

\subsection{An alternative description of the moduli space}
\label{sec:newhassett}

The construction of $\L^r_n$ in \cite{CDHLR} is as a subvariety of a ``Hassett space''---that is, a moduli space of stable rational curves with weighted marked points.  Roughly speaking, for any weight vector $\vec{w} = (w_1, \ldots, w_n) \in (\Q \cap (0,1])^n$ such that $\sum w_i > 2$, the associated genus-zero Hassett space $\M_{0,\vec{w}}$ is a moduli space of rational curves equipped with $n$ marked points, in which a subset of these marked points is allowed to coincide as long as the sum of their weights is at most one.  The stability condition on such curves is that, for each irreducible component with $n_0$ half-nodes and marked points in $I_0 \subseteq [n]$, one has
\[n_0 + \sum_{i \in I_0} w_i > 2.\]
Hassett constructed these moduli spaces in \cite{HASSETT2003316}, and moreover, he proved that if $w_i' \leq w_i$ for each $i$, then there is a birational weight-reduction morphism
\[\M_{0,\bf{w}} \rightarrow \M_{0, \bf{w}'}\]
whose exceptional locus can be expressed explicitly as a union of boundary divisors.

In addition to the inclusion into a Hassett space that arises from the construction of the moduli space, $\L^r_n$ carries another key morphism to a Hassett space, which is the quotient map $C \mapsto C/\sigma$.  The codomain of this map is the space $\M^1_n$ introduced in \cite[Section 3.1]{CDHLR}.  Namely, $\M^1_n = \M_{0,\w}$, where the weight vector is
\[\w=\left(\frac{1}{2}+\varepsilon,\frac{1}{2}+\varepsilon, 1, \underbrace{\varepsilon, \ldots, \varepsilon}_{n \text{ copies}}\right)\]
for any $0 < \varepsilon < 1/(2n+2)$.  A sample element of $\M^1_n$---which should be viewed as a single spoke of a curve in $\L^r_n$---is shown in Figure~\ref{fig:M16}.

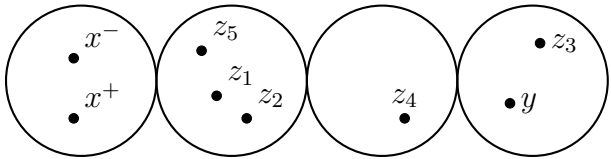
\begin{figure}[ht]
    \centering
    \begin{tikzpicture}
\draw[thick](0,0) circle (1);
\draw[thick](2,0) circle (1);
\draw[thick](4,0) circle (1);
\draw[thick](6,0) circle (1);
\filldraw (-0.1,0.3) circle[radius=1.8pt];
\filldraw (-0.1,-0.5) circle[radius=1.8pt];
\filldraw (1.6,0.4) circle[radius=1.8pt];
\filldraw (1.8,-0.2) circle[radius=1.8pt];
\filldraw (2.2,-0.5) circle[radius=1.8pt];
\filldraw (4.3,-0.5) circle[radius=1.8pt];
\filldraw (6.1,+0.5) circle[radius=1.8pt];
\filldraw (5.7,-0.3) circle[radius=1.8pt];
\node[above right] at (-0.1,0.3) {$x^-$};
\node[above right] at (-0.1,-0.5) {$x^+$};
\node[above right] at (1.6,0.4) {$z_5$};
\node[above right] at (1.8,-0.2) {$z_1$};
\node[above right] at (2.2,-0.5) {$z_2$};
\node[above] at (4.3,-0.5) {$z_4$};
\node[right] at (6.1,+0.5) {$z_3$};
\node[right] at (5.7,-0.3) {$y$};
\end{tikzpicture}
    \caption{A point of $\M^1_6$}
    \label{fig:M16}
\end{figure}

\begin{remark}
\label{rem:r=1}
As observed in \cite[Remark 8.1]{CDHLR}, the space $\M^1_n$ can alternatively be viewed as the result of setting $r=1$ in the definition of $\L^r_n$.
\end{remark}

For the purpose of realizing $\L^r_n$ as a wonderful compactification, we also require an analogue of the space $\M^1_n$ in which the points $z_i$  are allowed to coincide with $y$.  Specifically, let $X_0= \M_{0,\w_0}$ be the Hassett space with weight vector 
\[\w_0 := \left(\frac{1}{2}+\varepsilon, \frac{1}{2} + \varepsilon, 1-n\varepsilon, \underbrace{\varepsilon, \ldots, \varepsilon}_{n \text{ copies}}\right),\]
where $\varepsilon \in \Q$ is such that $0 < \varepsilon \leq 1/(2n+2)$.  Then
\[X_0 = (\P^1)^n,\]
since the weights ensure that the curves parameterized by $X_0$ consist of a single component.  Because $X_0$ differs from $\M^1_n$ only in that the weight on the marked point $y$ is reduced, there is a weight-reduction morphism
\[c: \M^1_n \rightarrow (\P^1)^n.\]
There is also an analogous morphism
\[b: \L^r_n \rightarrow (\P^1)^n,\]
which can be viewed as the composition of the forgetful map
\begin{align*}
     \L^r_n &\rightarrow \M^1_n\\
     (C; x^{\pm}, \{z_i^j\}, \{y^\ell\}) &\mapsto (C; x^{\pm}, \{z_i^0\}, y^0)
\end{align*}
with the map $c$.

\begin{remark}
\label{rem:lightLrn}
It is helpful---though not logically necessary---to view the codomain of $b$ as itself a moduli space, parameterizing analogous objects to those parameterized by $\L^r_n$ but in which all $n$ of the light orbits are allowed to coincide with the heavy orbit.  From this perspective, $b$ is also a weight-reduction morphism.
\end{remark}

Now, let
\[p: \L^r_n \rightarrow \M^1_n\]
be the morphism that sends an $(r,n)$-curve $C$ to the quotient of $C/\sigma$.  Then these morphisms fit together into a diagram
\begin{equation}
    \label{eq:Cartesian1}
\xymatrix{
\L^r_n\ar[r]^{b}\ar[d]_{p} & (\P^1)^n\ar[d]^{q}\\
\M^1_{n}\ar[r]_{c} & (\P^1)^n,
}
\end{equation}
where $q: (\P^1)^n \rightarrow (\P^1)^n$ is the ramified cover
\begin{equation}
    \label{eq:q}
q(p_1, \ldots, p_n) = (p_1^r, \ldots, p_n^r).
\end{equation}
See Figure~\ref{fig:Cartesian} for a depiction of the maps in this diagram.  

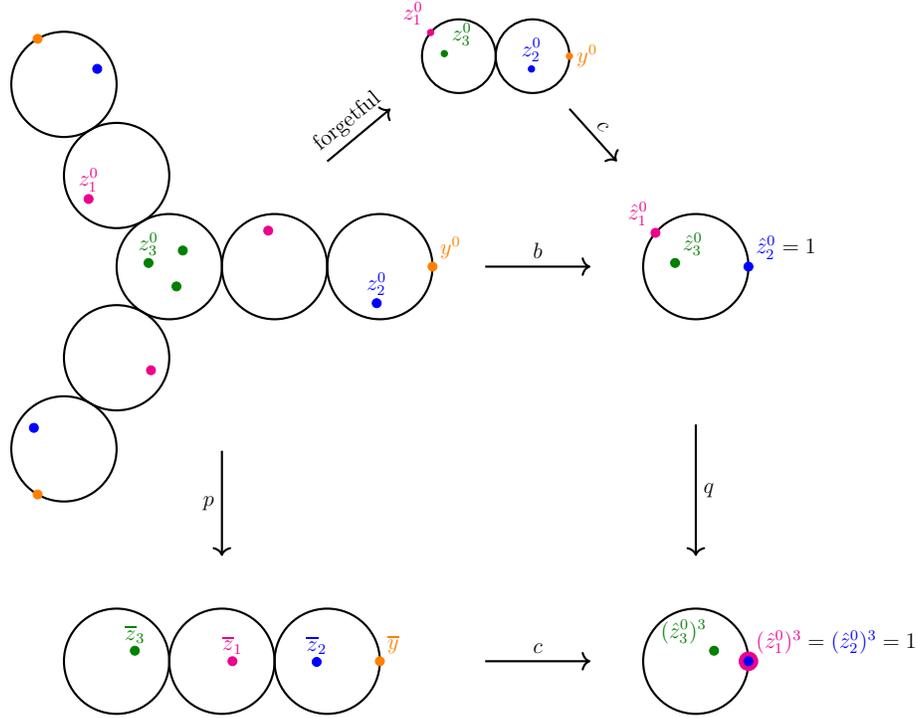
\begin{figure}[ht]
    \centering
\definecolor{ao(english)}{rgb}{0.0, 0.5, 0.0}
\begin{tikzpicture}[scale = 0.7]
\draw[thick] (0:0) circle (1);
\draw[thick] (0:2) circle (1);
\draw[thick] (120:2) circle (1);
\draw[thick] (240:2) circle (1);
\draw[thick] (0:4) circle (1);
\draw[thick] (120:4) circle (1);
\draw[thick] (240:4) circle (1);
\filldraw[color=ao(english)] (50:0.4) circle[radius=2.4pt];
\filldraw[color=magenta] (20:2) circle[radius=2.4pt];
\filldraw[color=blue] (-10:4) circle[radius=2.4pt];
\filldraw[color=orange] (0:5) circle[radius=2.4pt];
\filldraw[color=ao(english)] (170:0.4) circle[radius=2.4pt];
\filldraw[color=magenta] (140:2) circle[radius=2.4pt];
\filldraw[color=blue] (110:4) circle[radius=2.4pt];
\filldraw[color=orange](120:5) circle[radius=2.4pt];
\filldraw[color=ao(english)] (290:0.4) circle[radius=2.4pt];
\filldraw[color=magenta] (260:2) circle[radius=2.4pt];
\filldraw[color=blue] (230:4) circle[radius=2.4pt];
\filldraw[color=orange] (240:5) circle[radius=2.4pt];

\node[above, color=ao(english),scale=0.7] at (170:0.4) {${z_3^0}$};
\node[above ,color=magenta,scale=0.7] at  (140:2) {${z_1^0}$};
\node[above ,color=blue,scale=0.7] at (-10:4) {${z_2^0}$};
\node[above right ,color=orange,scale=0.7] at (0:5) {${y^0}$};

\draw [thick,->] (6,0) --(8,0)  node[pos=.5,above,scale=0.7] {$b$};
\begin{scope}[shift={(10,0)}]
\draw[thick] (0:0) circle (1);
\filldraw[color=ao(english)] (170:0.4) circle[radius=2.4pt];
\filldraw[color=magenta] (140:1) circle[radius=2.4pt];
\filldraw[color=blue] (0:1) circle[radius=2.4pt];
\node[above right, color=ao(english), scale=0.7] at (170:0.4) {$\hat{z}_3^0$};
\node[above left,color=magenta,scale=0.7] at  (140:1) {$\hat{z}_1^0$};
\node[above right,scale=0.7] at (0:1) {$\textcolor{blue}{\hat{z}_2^0}=1$};
\end{scope}
\draw [thick,->] (1,-3.5) --(1,-5.5)  node[pos=.5, left,scale=0.7] {$p$};
\begin{scope}[shift={(-1,-7.5)}]
\draw[thick] (0:0) circle (1);
\draw[thick] (0:2) circle (1);
\draw[thick] (0:4) circle (1);
\filldraw[color=ao(english)] (30:0.4) circle[radius=2.4pt];
\filldraw[color=magenta] (0:2.2) circle[radius=2.4pt];
\filldraw[color=blue] (-0.2:3.8) circle[radius=2.4pt];
\filldraw[color=orange] (0:5) circle[radius=2.4pt];

\node[above, color=ao(english),scale=0.7] at (30:0.4) {$\overline{z}_3$};
\node[above,color=magenta,scale=0.7] at  (0:2.2) {$\overline{z}_1$};
\node[above,color=blue,scale=0.7] at (-0.2:3.8) {$\overline{z}_2$};
\node[above right,color=orange,scale=0.7] at (0:5) {$\overline{y}$};

\end{scope}
\draw [thick,->] (6,-7.5) --(8,-7.5)  node[pos=.5,above,scale=0.7] {$c$};
\draw [thick,->] (10,-3) --(10,-5.5) node[pos=.5, right,scale=0.7] {$q$};
\begin{scope}[shift={(10,-7.5)}]
\draw[thick] (0:0) circle (1);
\filldraw[color=ao(english)] (30:0.4) circle[radius=2.4pt];
\filldraw[color=magenta] (0:1) circle[radius=5pt];
\filldraw[color=blue] (0:1) circle[radius=2.4pt];
\node[above left, color=ao(english),scale=0.7] at (30:0.4) {$(\hat{z}_3^0)^3$};
\node[above right,scale=0.7] at  (0:1) {$\textcolor{magenta}{(\hat{z}_1^0)^3}=\textcolor{blue}{(\hat{z}_2^0)^3}=1$};
\end{scope}
\begin{scope}[shift={(5.5,4)},scale=0.7]
\draw[thick] (0:0) circle (1);
\draw[thick] (0:2) circle (1);
\filldraw[color=blue] (-10:2) circle[radius=2.4pt];
\filldraw[color=orange] (0:3) circle[radius=2.4pt];
\filldraw[color=ao(english)] (170:0.4) circle[radius=2.4pt];
\filldraw[color=magenta] (140:1) circle[radius=2.4pt];

\node[above right, color=ao(english),scale=0.7] at (170:0.4) {$z_3^0$};
\node[above left,color=magenta,scale=0.7] at  (140:1) {$z_1^0$};
\node[above,color=blue,scale=0.7] at (-10:2) {$z_2^0$};
\node[right,color=orange,scale=0.7] at (0:3) {$y^0$};

\end{scope}
\draw [thick,->] (3,2) --(4.2,3) node[pos=.5,sloped,above,scale=0.7] {forgetful};
\draw [thick,->] (7.6,3) --(8.5,2) node[pos=.5,sloped,above, scale = 0.7] {$c$};
\end{tikzpicture}
    \caption{A representation of the maps in diagram \eqref{eq:Cartesian1} in the case where $r=n=3$, with the points $x^{\pm}$ omitted for clarity.  In the upper-right corner, the three coordinates in $(\P^1)^3$ are $\zeta$, $1$, and a point $\hat{z}_3^0$ that is not a $3$rd root of unity.  In the lower-right, the three coordinates are $1$, $1$, and $(\hat{z}_3^0)^r$.}
    \label{fig:Cartesian}
\end{figure}

In fact, \eqref{eq:Cartesian1} is Cartesian.  Heuristically, this makes sense: a curve in $\M^1_n$ specifies a single spoke of a curve in $\L^r_n$, which determines the entire element of $\L^r_n$ modulo the ordering of the points within each orbit, while a point in $(\P^1)^n$ determines the choice of which point within each orbit shall be labeled $z_i^0$.  We make this argument precise in the following lemma.

\begin{lemma}
\label{lem:Cartesian} The diagram \eqref{eq:Cartesian1} is Cartesian.
\end{lemma}
\begin{proof} Let $B$ be any scheme, and suppose we are given morphisms $\rho: B \rightarrow \M^1_n$ and $\beta: B \rightarrow (\P^1)^n$ such that the diagram
\begin{equation*}
\xymatrix{
B\ar[r]^{\beta}\ar[d]_{\rho} & (\P^1)^n\ar[d]^{q}\\
\M^1_{n}\ar[r]_{c} & (\P^1)^n
}
\end{equation*}
commutes.  Our goal is to construct a map $B \rightarrow \L^r_n$, or in other words, a family of $(r,n)$-curves over $B$.

First, note that from the definition of $\M^1_n$ as a moduli space, the map $\rho$ induces a family $\pi^1_n: \mathcal{C}^1_n \rightarrow B$ of weighted-pointed curves over $B$, with sections $x^{\pm}, z_1, \ldots, z_n$, and $y$.  The map $c \circ \rho$ also induces a family of weighted-pointed curves over $B$, namely the family
\begin{equation}
    \label{eq:family1}
    \xymatrix{
B \times \P^1\ar[d]_{\pi_B}\\
B\ar@/^-1.5pc/_{\overline{x}^{\pm}, \overline{z}_1, \ldots, \overline{z}_n, \overline{y}}[u]}
\end{equation}
where the sections $\overline{x}^{\pm}, \overline{z}_1, \ldots, \overline{z}_n, \overline{y}$ are defined by
\begin{align*}
    &\overline{x}^+(b) = (b,\infty)\\
    &\overline{x}^-(b) = (b,0)\\
    &\overline{y}(b) = (b,1)\\
    &\overline{z}_i(b) = (b, \; (c \circ \rho)_i(b)),
\end{align*}
where $(c \circ \rho)_i(b) \in \P^1$ denotes the $i$th coordinate of $(c\circ \rho)(b) \in (\P^1)^n$.  Since the map $c: \M^1_n \rightarrow (\P^1)^n$ is a weight-reduction morphism between Hassett spaces, it can be upgraded to the level of families, yielding a morphism
\[\widetilde{c}: \mathcal{C}^1_n \rightarrow B \times \P^1\]
that takes the sections of $\mathcal{C}^1_n$ to the corresponding sections of $B \times \P^1$.

Next, note that the map $\beta$ also induces a family of weighted-pointed curves.  Taking the perspective of Remark~\ref{rem:lightLrn}, we view the family induced by $\beta$ as
\begin{equation}
    \label{eq:family2}
\xymatrix{
B \times \P^1\ar[d]_{\pi_B}\\
B\ar@/^-1.5pc/_{\hat{x}^{\pm}, \{\hat{z}_i^j\}, \{\hat{y}^\ell\}}[u]},
\end{equation}
where $\hat{x}^{\pm}=(\overline{x}^\pm)^r = \overline{x}^\pm$, and the remaining sections are defined by
\begin{align*}
    &\hat{y}^\ell(b) = (b,\zeta^\ell)\\
    &\hat{z}_i^j(b) = (b, \zeta^j \beta_i(b))
\end{align*}
for $\ell, j \in \Z_r$ and $i \in [n]$; note that this is a family of curves with marked points of weights
\[\left(\frac{1}{2}+\varepsilon, \frac{1}{2} + \varepsilon, \underbrace{1-n\varepsilon, \ldots, 1-n\varepsilon}_{r \text{ copies}}, \underbrace{\varepsilon, \ldots, \varepsilon}_{rn \text{ copies}}\right).\]
Since both \eqref{eq:family1} and \eqref{eq:family2} are trivial families, the morphism $q$ can be upgraded to a morphism between them: namely, we have
 \[\widetilde{q}: B \times \P^1 \rightarrow B \times \P^1\]
given by $\widetilde{q}(b,p) = (b,p^r)$, which fixes the sections $\overline{x}^{\pm}$ and takes $\hat{y}^\ell$ to $\overline{y}$ as well as $\hat{z}_i^j$ to $\overline{z}_i$ for each $i,j,$ and $\ell$.

Now, to produce the requisite family of $(r,n)$-curves, define $\mathcal{C}^r_n$ as the fiber product of the diagram
\begin{equation}
    \label{eq:CartesianCurves}
\xymatrix{
\mathcal{C}^r_n \ar[r]^-{\widetilde{b}}\ar[d]_{\widetilde{p}} & B \times \P^1\ar[d]^{\widetilde{q}}\\
\mathcal{C}^1_n \ar[r]_-{\widetilde{c}} & B \times \P^1.
}
\end{equation}
We claim, first, that $\mathcal{C}^r_n$ is a flat family of curves over $B$.  It is certainly equipped with a map $\pi: \mathcal{C}^r_n \rightarrow B$, namely
\[\pi := \widetilde{p} \circ \pi^1_n = \widetilde{b} \circ \pi_B.\]
To see that $\pi$ is flat, note that $\widetilde{q}$ is \'etale away from $B \times \{0,\infty\}$, so, since \'etaleness is preserved by base change, it follows that $\widetilde{p}$ is \'etale on $\mathcal{C}^r_n \setminus \widetilde{b}^{-1}(B \times \{0,\infty\})$.  In particular, then, the restriction of $\pi$ to this locus is the composition of an \'etale morphism with the flat morphism $\pi^1_n$, so it is flat.  On the other hand, the map $\widetilde{c}$ is an isomorphism away from $\widetilde{c}^{-1}(B \times \{1\})$, so it follows that $\widetilde{b}$ is an isomorphism on $\mathcal{C}^r_n \setminus \widetilde{p}^{-1}(\widetilde{c}^{-1}(B \times \{1\}))$.  As a result, the restriction of $\pi$ to this locus is the composition of an isomorphism with the flat morphism $\pi_B$, so it is flat.  Having covered $\mathcal{C}^r_n$ by open sets on which $\pi$ is flat, we conclude that $\mathcal{C}^r_n$ is indeed a flat family of curves over $B$.

In order to make $\mathcal{C}^r_n$ into a family of $(r,n)$-curves, we must equip it with an order-$r$ automorphism and sections. For the first of these, let 
\[\overline{\sigma}: B \times \P^1 \rightarrow B \times \P^1\]
be the automorphism $\overline{\sigma}(b,p) = (b, \zeta p)$.  Then we have a diagram
\[\xymatrix{
\mathcal{C}^r_n \ar[r]^-{\overline{\sigma} \circ \widetilde{b}} \ar[d] & B \times \P^1\ar[d]^{\widetilde{q}}\\
\mathcal{C}^1_n \ar[r]_-{\widetilde{c}} & B \times \P^1,
}\]
and the universal property of $\mathcal{C}^r_n$ as a fiber product yields a morphism $\sigma: \mathcal{C}^r_n \rightarrow \mathcal{C}^r_n$ that is easily confirmed to be an order-$r$ automorphism over $B$.

The construction of the sections is similar; in particular, by the universal property of fiber products, a section of $\mathcal{C}^r_n$ is determined by sections of $\mathcal{C}^1_n$ and $B \times \P^1$.  We define $x^\pm$ as the section determined by the section $x^\pm$ of $\mathcal{C}^1_n$ and $\overline{x}^{\pm}$ of $B \times \P^1$, define $y^\ell$ as the section determined by $y$ and $\hat{y}^\ell$, and define $z_i^j$ as the section determined by $z_i$ and $\hat{z}_i^j$.  From here, it is straightforward to check that each fiber
\[(\pi^{-1}(b); x^{\pm}(b)), \{z_i^j(b)\}, \{y^\ell(b)\})\]
of $\pi$ is indeed a stable $(r,n)$-curve.  Thus, we have given $\mathcal{C}^r_n$ the structure of an $(r,n)$-curve over $B$, meaning that we have a map $B \rightarrow \L^r_n$.  By construction, this map makes the diagram
\[\xymatrix{
B \ar[dr]\ar@/^+1.5pc/[drr]^{\beta}\ar@/^-1.5pc/[ddr]_{\rho} & & \\
& \L^r_n\ar[r]^{b}\ar[d]_{p} & (\P^1)^n\ar[d]^{q}\\
& \M^1_{n}\ar[r]_{c} & (\P^1)^n
}\]
commute, so the proof is complete.
\end{proof}

\subsection{The moduli space as a wonderful compactification}

We are now prepared to describe how $\L^r_n$ arises as a wonderful compactification.  The ambient variety is $(\P^1)^n$, and in this variety, we consider the arrangement consisting of the hyperplanes
\begin{equation}
\label{eq:Hij}\widetilde{H}_i^j = \{(p_1, \ldots, p_n) \in (\P^1)^n \; | \; p_i = \zeta^j\}
\end{equation}
for each $i \in [n]$ and $j \in \Z_r$.  Note that this is the product arrangement induced by $n$ copies of the hyperplane arrangement
\begin{equation}
    \label{eq:Ar}
    \mathcal{A}_r := \Big\{\{1\}, \{\zeta\}, \{\zeta^2\}, \ldots, \{\zeta^{r-1}\}\Big\}
\end{equation}
in $\P^1$, where $\zeta$ is our fixed primitive $r$th root of unity.

\begin{theorem}
\label{thm:Lrnwc}
For any $r \geq 2$ and $n \geq 0$, the moduli space $\L^r_n$ is the wonderful compactification of the arrangement 
\[\{\widetilde{H}_i^j\}_{i \in [n], \; j \in \Z_r}\]
in $(\P^1)^n$, with maximal building set.
\end{theorem}

\begin{proof}
Our goal is to realize $\L^r_n$ as an iterated blow-up of $(\P^1)^n$ as described in Section~\ref{subsec:WCbackground}, and the first key observation is that for $\M^1_n$, the analogous result holds. Specifically, for any $k \in \{0,1, \ldots, n\}$, let $X_k = \M_{0,\w_k}$ be the Hassett space with weight vector
\[\w_k := \left(\frac{1}{2}+\varepsilon, \frac{1}{2} + \varepsilon, 1-(n-k)\varepsilon, \underbrace{\varepsilon, \ldots, \varepsilon}_{n \text{ copies}}\right),\]
where, once again, $\varepsilon \in \Q$ is such that $0< \varepsilon \leq 1/(2n+2)$; this space parameterizes the same objects as $\M^1_n$, but in which $n-k$ of the light points $z_i$ are allowed to coincide with $y$.  When $k=0$, we obtain the space $X_0$ described in the previous section, which can be identified with $(\P^1)^n$; and when $k=n$, we obtain $X_n = \M^1_n$.

Each of the spaces $X_{k+1}$ is obtained from $X_k$ by blow-up along a smooth subvariety.  Indeed, if we let $Z_k \subseteq X_k$ be the locus where $n-k$ of the points $z_i^j$ coincide with $y$, then
 \begin{itemize}
     \item $X_1$ is the blow-up of $X_0$ along $Z_0$,
     \item $X_2$ is the blow-up of $X_1$ along the proper transform of $Z_1$,
     \item $X_3$ is the blow-up of $X_2$ along the proper transform of $Z_2$,
 \end{itemize}
and so on.  The proofs of these statements follow from \cite[Theorem 4.8]{AG}, which shows that the weight-reduction morphism $c_k: X_{k+1} \rightarrow X_k$ is a blow-up when the change of weights is a ``simple'' wall-crossing (see \cite[Definition 4.1]{AG}), which is true in this case.

Now, we inductively define spaces $Y_k$ with maps $q_k: Y_k \rightarrow X_k$, for each $k \in \{0,1,\ldots, n\}$, as follows.  When $k=0$, set $Y_k = (\P^1)^n$, and set $q_0: Y_0 \rightarrow X_0$ to be the map $(\P^1)^n \rightarrow (\P^1)^n$ given by \eqref{eq:q}.  Then, having defined $Y_k$ and $q_k$, define $Y_{k+1}$ and $q_{k+1}$ by the following Cartesian diagram:
\begin{equation}
    \label{eq:Cartesian}
\xymatrix{
Y_{k+1} \ar[r]^{b_k}\ar[d]_{q_{k+1}} & Y_k\ar[d]^{q_k}\\
X_{k+1} \ar[r]_{c_k} & X_k.
}
\end{equation}
Note that each $q_k$ is flat (since $q= q_0$ is flat and \eqref{eq:Cartesian} is Cartesian), so since blow-ups commute with flat base change (see \cite[Exercise 24.2.P]{VakilBook}), the fact that $X_{k+1}$ is the blow-up of $X_k$ along $Z_k$ implies that $Y_{k+1}$ is the blow-up of $Y_k$ along $q_k^{-1}(Z_k)$.

Since $Y_0= (\P^1)^n$ and $Y_n = \L^r_n$ by Lemma~\ref{lem:Cartesian}, we have now shown that $\L^r_n$ is obtained from $(\P^1)^n$ by the following sequence of blow-ups:
\begin{itemize}
    \item blow up $(\P^1)^n$ along $q_0^{-1}(Z_0)$, which is the union of the points where all $n$ coordinates are equal to $r$th roots of unity;
    \item blow up along $q_1^{-1}(Z_1)$, which is the proper transform of the union of the lines in $(\P^1)^n$ where $n-1$ coordinates are equal to $r$th roots of unity;
    \item blow up along $q_2^{-1}(Z_2)$, which is the proper transform of the union of the planes in $(\P^1)^n$ where $n-2$ coordinates are equal to $r$th roots of unity;
\end{itemize}
and so on.  In other words, we are iteratively blowing up $(\P^1)^n$ along all intersections of the hyperplanes \eqref{eq:Hij}, in increasing order with respect to inclusions.  This is precisely the construction of the wonderful compactification of this arrangement (with its maximal building set), so the proof is complete.
\end{proof}

Observe that by Remark~\ref{rem:r=1}, $\M^1_n$ can be viewed as the $r=1$ case of the space $\Lrn$. Thus, the first part of the above proof can be interpreted as showing that, also in this limit case, $\L^1_n$ is an iterated blow-up of $(\P^1)^n$ and can be seen as a wonderful compactification for a non-essential hyperplane arrangement.

\section{The Chow ring of \texorpdfstring{$\Lrn$}{the moduli space}}
\label{sec:chowring}

The presentation of $\L^r_n$ as a wonderful compactification via Theorem~\ref{thm:Lrnwc}, together with the result of Theorem~\ref{thm:wcproduct}, allows us to calculate $A^*(\L^r_n)$, and the goal of this section is to carry out this computation explicitly.

\subsection{The nested set fan for \texorpdfstring{$\L^r_n$}{the moduli space}}

By Theorem~\ref{thm:wcproduct}, the Chow ring of a wonderful compactification is determined by its nested set fan.  Our first goal, then, is to describe the nested set fan of the arrangement
\begin{equation*}
    \label{eq:A}
\calA = \{\widetilde{H}_i^j\}_{i \in [n], j \in \Z_r}
\end{equation*}
in $(\P^1)^n$ given by \eqref{eq:Hij}, with its maximal building set $\calG = \calL_{\calA} \setminus \{\hat{0}\}$.

We require two pieces of combinatorial terminology, both of which appeared in \cite{CDHLR}.

\begin{definition}
\label{def:decoratedchain}
A {\bf $\Z_r$-decorated subset of $[n]$} is a pair $\tI = (I, \a)$, in which $I \subseteq [n]$ is a nonempty subset and $\a: I \rightarrow \Z_r$ is any function.  More generally, a {\bf $\Z_r$-decorated chain of subsets of $[n]$} (or simply {\bf chain}, for short) is a tuple
\[\tbI = (I_1, \ldots, I_\ell, \a),\]
where
\[\emptyset = I_0 \subsetneq I_1 \subsetneq \cdots \subsetneq I_\ell \subseteq [n]\]
and
\[\a: I_\ell \rightarrow \Z_r.\]
We refer to the number $\ell$ as the {\bf length} of the chain.
\end{definition}

From the definition of the hyperplanes $\widetilde{H}_i^j$ in \eqref{eq:Hij}, one sees that the intersection $\widetilde{H}_i^j \cap \widetilde{H}_i^{j'}$ is empty unless $j = j'$, whereas all of the intersections $\widetilde{H}_i^j \cap \widetilde{H}_{i'}^{j'}$ with $i \neq i'$ are nonempty.  It follows that the elements of the intersection lattice $\calL_{\calA}$ are precisely the intersections
\[H_{\tI} := \bigcap_{i \in I} \widetilde{H}_i^{-\a(i)} \]
for each decorated set $\tI = (I,\a)$. 

\begin{remark}
The negative exponents in the definition of $H_{\tI}$ may look strange at a glance, but this convention is chosen for consistency with the indexing of boundary strata by chains in \cite{CDHLR}; see Remark~\ref{rem:negative} below.
\end{remark}

Given that  $\calG$ is the maximal building set, the $\calG$-nested sets are simply chains in $\calL_{\calA} \setminus \{\hat{0}\}$ as a poset.  The ordering on $\calL_{\calA}$ is by reverse inclusion, and from this one sees that
\[H_{\tilde{I}} \leq H_{\tilde{J}} \; \text{ if and only if } \; \tI \leq \widetilde{J},\]
where the ordering on decorated sets is given by
\[(I, \a) \leq (J,\b)  \; \text{ if and only if } \;I \subseteq J \text{ and } \a(i) = \b(i) \;\; \text{ for all } i \in I.\]
As a result, the $\calG$-nested sets are indexed by chains in the sense of Definition~\ref{def:decoratedchain}: namely, if $\tbI = (I_1, \ldots, I_\ell, \a)$ is a chain, then the corresponding $\calG$-nested set is
\[H_{\widetilde{I_1}} \leq H_{\widetilde{I_2}} \leq \cdots \leq H_{\widetilde{I_\ell}}.\]
Comparing this to Definition~\ref{def:nestedsetfan}, we see that the nested set fan for $(\calL_\calA, \calG)$, which we denote by $\Sigma^r_n$, can be described as follows.

\begin{definition}\label{def:sigma}
Let
\[V_{\calA}= (\R^r/\R)^{\oplus n},\]
and denote the images of the standard basis vectors in the $i$th copy of $\R^r/\R$ by $e_i^0, \ldots, e_i^{r-1}$.  Then $\Sigma^r_n$ is the fan in $V_{\calA}$ with a cone
\[\sigma_{\tbI} := \text{Cone}\left\{ \sum_{i \in I_1} e_i^{-\a(i)}, \ldots, \sum_{i \in I_\ell} e_i^{-\a(i)}\right\}\]
for each chain $\tbI$.  See Figure~\ref{fig:Sigmarn} for an illustration.
\end{definition}

\begin{figure}[ht]
    \centering
\definecolor{ao(english)}{rgb}{0.0, 0.5, 0.0}
\begin{tikzpicture}[xscale=0.6, yscale=0.6]
\draw [fill=gray!30!white,gray!30!white] (-2.4,-2.4) rectangle (2.4,2.4);
\draw [ao(english),fill=ao(english),opacity=0.3] (0,0) -- (0,2.4) -- (-2.4,2.4)  -- cycle;
\draw [ultra thick,->] (0,0) --(3,0);
\draw [ultra thick,->] (0,0) --(-3,0);
\draw [ao(english),ultra thick,->] (0,0) --(0,3);
\draw [ultra thick,->] (0,0) --(0,-3);
\draw [ultra thick,->] (0,0) --(2.7,2.7);
\draw [ao(english),ultra thick,->] (0,0) --(-2.7,2.7);
\draw [ultra thick,->] (0,0) --(2.7,-2.7);
\draw [ultra thick,->] (0,0) --(-2.7,-2.7);
\filldraw [ao(english)] (0,0) circle[radius=1mm];
\node at (0, -3.6) {$e_2^1$};
\node at (0, 3.6) {$e_2^0$};
\node at (3.6,0) {$e_1^0$};
\node at (-3.6,0) {$e_1^1$};
\node [ao(english)] at (-0.7,1.6) {$\sigma_{\tbI}$};
\end{tikzpicture}
\caption{The fan $\Sigma^2_2$.  The cone $\sigma_{\tbI}$ labeled in green corresponds to the chain $\tbI = (\{2\} \subseteq \{1,2\}, \; \a)$ in which $\a(1) =1$ and $\a(2)=0$.}
    \label{fig:Sigmarn}
\end{figure}
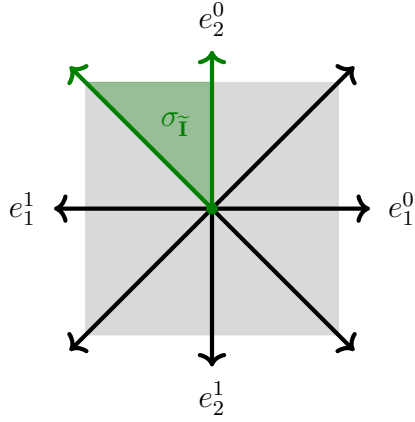

\begin{remark}
\label{rem:intersection}
The intersection $\sigma_{\tbI} \cap \sigma_{\tbJ}$ is the cone $\sigma_{\tbI \cap \tbJ}$, where $\tbI \cap \tbJ$ is the following chain. Let $\tbI=(I_1, \dots, I_{\ell_I}, \a)$ and $\tbJ=(J_1, \dots, J_{\ell_J}, \mathfrak{b})$, and define
\[(I\cap J)_{i,j}=\{ k \in I_i \cap I_j \; |\;  \a(k) = \b(k)\}.\]
The collection of subsets $(I\cap J)_{i,j}$ with $i \in [\ell_I]$ and $j \in [\ell_J]$ can be reordered to define a chain of subsets of $[n]$ such that the biggest one, given by $(I\cap J)_{\ell_I,\ell_J}$, admits a unique map to $\Z_r$ restricting $\a$ (or, equivalently, $\b$).
\end{remark}

\begin{remark}
\label{rem:stellar}
An alternative way to construct $\Sigma^r_n$, by Lemma~\ref{lem:stellar}, is via a stellar subdivision procedure.  Specifically, let $\Sigma_r$ be the nested set fan for the arrangement \eqref{eq:Ar} in $\P^1$ with its maximal building set; this is a $1$-dimensional fan in $\R^r/\R$ with $r$ rays spanned by the images of the standard basis vectors in $\R^r$.  Then the Cartesian product $(\Sigma_r)^{\times n}$ is a fan in $V_{\calA}$.  Recalling that $V_{\calA}$ has a vector
\[v_G := \sum_{\widetilde{H}^j_i \supseteq G} e_i^j\]
for each $G \in \calL_{\calA} \setminus \{\hat{0}\}$, the content of Lemma~\ref{lem:stellar} is that $\Sigma^r_n$ can be obtained from $(\Sigma_r)^{\times n}$ by successive stellar subdivision along the vectors $v_{H_{\tI}}$ for each nested set $\tI$ with $|I|>1$, in inclusion-increasing order with respect to the varieties $H_{\tI}$.  We illustrate this construction in an example in Figure~\ref{fig:stellar}.
\end{remark}

\begin{figure}[ht]
    \centering
    \begin{tikzpicture}[xscale=0.5, yscale=0.5]
\filldraw (0,0) circle[radius=1mm];
\draw [ultra thick,->] (0,0) --(3,0);
\draw [ultra thick,->] (0,0) --(-3,0);
\node at (0, -4) {$\Sigma_2$};

\begin{scope}[shift={(8,0)}]
\draw [fill=gray!30!white,gray!30!white] (-2.4,-2.4) rectangle (2.4,2.4);
\filldraw (0,0) circle[radius=1mm];
\draw [ultra thick,->] (0,0) --(3,0);
\draw [ultra thick,->] (0,0) --(-3,0);
\draw [ultra thick,->] (0,0) --(0,3);
\draw [ultra thick,->] (0,0) --(0,-3);
\node at (0, -4) {$(\Sigma_2)^{\times 2}$};
\end{scope}

\begin{scope}[shift={(16,0)}]
\draw [fill=gray!30!white,gray!30!white] (-2.4,-2.4) rectangle (2.4,2.4);
\draw [ultra thick,->] (0,0) --(3,0);
\draw [ultra thick,->] (0,0) --(-3,0);
\draw [ultra thick,->] (0,0) --(0,3);
\draw [ultra thick,->] (0,0) --(0,-3);
\draw [ultra thick,->] (0,0) --(2.7,2.7);
\draw [blue, ultra thick,->] (0,0) --(-2.7,2.7);
\draw [ultra thick,->] (0,0) --(2.7,-2.7);
\draw [ultra thick,->] (0,0) --(-2.7,-2.7);
\filldraw [blue] (0,0) circle[radius=1mm];

\node at (0, -4) {$\Sigma^2_2$};
\node [above left, blue] at (-2.7,2.7) {$v_{H_{\widetilde{I}}}$};

\end{scope}

\end{tikzpicture}
    \caption{The fan $\Sigma_2^2$, obtained via stellar subdivision from the Cartesian product of two copies of the fan $\Sigma_2$.  The labeled vector $v_{H_{\tI}}$ corresponds to the nested set $\tI = (\{1,2\}, \a)$ in which $\a(1) = 1$ and $\a(2)=0$.}
    \label{fig:stellar}
\end{figure}
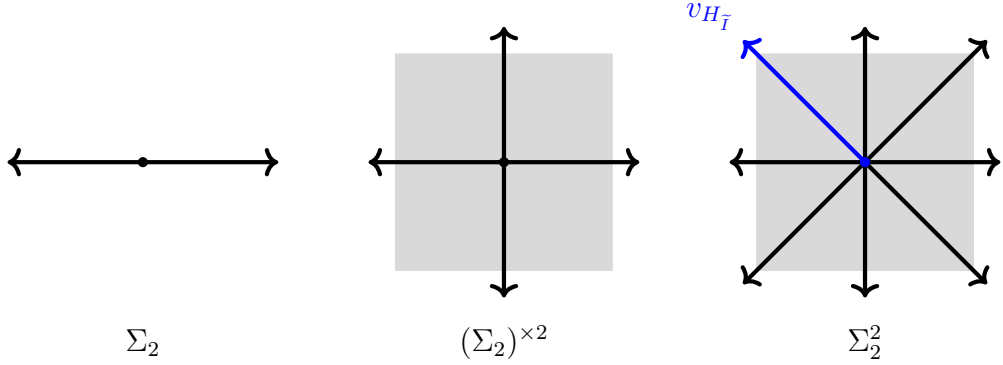

In light of the description of $\Sigma^r_n$ in Definition~\ref{def:sigma}, the torus-invariant strata in $X_{\Sigma^r_n}$ can be indexed by chains $\tbI$.  On the other hand, we proved in \cite{CDHLR} that the boundary strata of $\L^r_n$ are also indexed by chains, and in fact, the next section shows that the inclusion
\[\L^r_n \hookrightarrow X_{\Sigma^r_n}\]
provided by Theorem~\ref{thm:Lrnwc} matches these two types of strata with one another.  Before stating this result, we must recall the association of boundary strata with chains from \cite{CDHLR}.

\subsection{Boundary strata and chains}
\label{subsec:boundarystrata}

In order to describe the boundary strata of $\L^r_n$, we first explain how components of an $(r,n)$-curve are labeled.  Let $(C; \sigma; x^{\pm}, \{y^k\}, \{z_i^j\})$ be a stable $(r,n)$-curve, and suppose that $C$ has ``length'' $\ell$ in the sense that each of its $r$ spokes (chains of $\P^1$'s emanating from the central component) consists of $\ell$ components.  Then, for each $k \in \Z_r$, we denote the components of the spoke containing $y^k$ by
\[C^k_1, C^k_2, \ldots, C^k_\ell,\]
where $y^k \in C^k_1$ and the other components are labeled in order from outermost to innermost.  We denote the central component by $C_{\ell+1}$.

Given this labeling, the idea of the association of a boundary stratum to a chain is that the outermost components $\{C^k_1\}_{k \in \Z_r}$ contain the marked points indexed by $I_1$ (in an order dictated by $\a$), the next-outermost components $\{C^k_2\}_{k \in \Z_r}$ contain the marked points indexed by $I_2 \setminus I_1$, and so on, until $[n] \setminus I_{\ell}$, which indexes the marked points on the central component.  More precisely, the association is as follows.

\begin{definition}
\label{def:SI}
Let $\tbI = (I_1, \ldots, I_\ell, \a)$ be a chain. We say that $(C; \sigma\; x^{\pm}, \{y^k\}, \{z_i^j\}) \in \L^r_n$ is {of type $\tbI$} if $C$ is an $r$-pinwheel curve of length $\ell$ and, using the above notation, we have
\begin{enumerate}
    \item for each $j \in \{1, \ldots, \ell\}$, the light marked points on $C^0_j$ are precisely
    \[\{z_i^{\a(i)} \; | \; i \in I_j \setminus I_{j-1}\},\]
    where $I_0:= \emptyset$;
    \item the light marked points on the central component $C_{\ell+1}$ are
    \[\{z_i^k \; | \; i \in [n] \setminus I_\ell, \; k \in \Z_r\} \cup \{x^{\pm}\}.\]
\end{enumerate}
We define the boundary stratum $S_{\tbI} \subseteq \L^r_n$ to be the closure of the locus of curves of type $\tbI$. \end{definition}

The $(3,4)$-curve of Figure~\ref{fig:samplecurve}, for example, is a generic element of the boundary stratum $S_{\tbI}$ in which $\tbI = (I_1, I_2, \a)$ for
\[I_1 = \{3\}, \;\; I_2 = \{2,3,4\}\]
and $\a: I_2 \rightarrow \Z_3$ given by
\[\a(2) = 1, \;\; \a(3) = 0, \;\; \a(4) = 2.\]

\begin{remark}
\label{rem:negative} The first condition in Definition~\ref{def:SI} implies that, for an $(r,n)$-curve of type $\tbI$, the light marked point $z_i^0$ is on the same spoke of $C$ as $y^{-\a(i)}$.  Given that the positions of all other light marked points are determined by the location of the points $z_i^0$, this helps to explain why $-\a(i)$ appears in the definitions of $H_{\tI}$ and $\sigma_{\tbI}$ above.
\end{remark}

We proved in \cite[Proposition 5.4]{CDHLR} that the association $\tbI \mapsto S_{\tbI}$ is a bijection from chains to boundary strata in $\L^r_n$, and that under this bijection, the codimension of $S_{\tbI}$ corresponds to the length of $\tbI$ whereas an inclusion of boundary strata $S_{\tbI} \subseteq S_{\tbJ}$ corresponds to the statement that $\tbI$ ``refines'' $\tbJ$ in the sense of \cite[Definition 4.2]{CDHLR}.  In particular, the boundary divisors are associated to chains of length $1$, which are $\Z_r$-decorated subsets of $[n]$.  We denote by
\[D_{\tI} \subseteq \L^r_n\]
the boundary divisor corresponding to the decorated set $\tI = (I,\a)$.

Now, returning to the fan $\Sigma^r_n$ of Definition~\ref{def:sigma}, for any chain $\tbI$, denote by $X_{\tbI} \subseteq X_{\Sigma^r_n}$ the torus-invariant stratum associated to the cone $\sigma_{\tbI}$ of $\Sigma^r_n$.  Then we have the following correspondence between the strata $X_{\tbI}$ and the strata $S_{\tbI}$.

\begin{proposition}
\label{prop:strata}
Under the inclusion $\L^r_n \hookrightarrow X_{\Sigma^r_n}$ given by Theorems~\ref{thm:wcproduct} and \ref{thm:Lrnwc}, the pullback of the torus-invariant stratum $X_{\tbI}$ is the boundary stratum $S_{\tbI}$.  In particular, the pullback of the torus-invariant divisor $X_{\tI}$ is the boundary divisor $D_{\tI}$.
\end{proposition}
\begin{proof}
It suffices to prove the claim for divisors, since any torus-invariant stratum (respectively, boundary stratum) is the intersection of the torus-invariant divisors (respectively, boundary divisors) that contain it, and in both cases, the intersection of the stratum indexed by $\tbI$ and the stratum indexed by $\tbJ$ is the stratum indexed by the chain $\tbI \cap \tbJ$ described in Remark~\ref{rem:intersection}.  Thus, we fix a decorated set $\tI = (I, \a)$ and consider the corresponding boundary divisor $D_{\tI} \subseteq \L^r_n$.

From the last paragraph of the proof of Theorem~\ref{thm:Lrnwc}, one can view $\L^r_n$ as an iterated blow-up
\[\L^r_n = Y_n \longrightarrow Y_{n-1} \longrightarrow \cdots \longrightarrow Y_1 \longrightarrow Y_0 = (\P^1)^n,\]
where $Y_{k+1}$ is obtained from $Y_k$ by blow-up along the proper transform of the locus
\[W_k:= \bigcup_{\widetilde{J} \; \big| \; |J| =n-k} H_{\widetilde{J}} \subseteq (\P^1)^n.\]
If $E_{\tI} \subseteq Y_{n-|I|+1}$ denotes the exceptional divisor over $H_{\tI}$, then from this perspective, $D_{\tI}$ is the proper transform in $\L^r_n$ of $E_{\tI}$.

On the other hand, one can also view $X_{\Sigma^r_n}$ as an iterated blow-up, by the stellar subdivision perspective of Lemma~\ref{lem:stellar}.  Namely, let $\Sigma_r$ be the nested set fan for the arrangement \eqref{eq:Ar}, as described in Remark~\ref{rem:stellar}.  Then $\Sigma_r$ is obtained from the fan for $\P^{r-1}$ by removing all but the $1$-dimensional cones, so
\[X_{\Sigma_r} = \P^{r-1} \setminus \bigcup_{j \neq \ell} (\widehat{H}_j \cap \widehat{H}_\ell),\]
where $\widehat{H}_j \subseteq \P^{r-1}$ denotes the $j$th coordinate hyperplane; in other words, a point of $\P^{r-1}$ belongs to $X_{\Sigma^r_n}$ if and only if at most one of its coordinates is zero.  Thus,
\[X_{(\Sigma_r)^{\times n}} = \left(\P^{r-1} \setminus \bigcup_{j \neq \ell} (\widehat{H}_j \cap \widehat{H}_\ell)\right)^n,\]
and Lemma~\ref{lem:stellar} says that $X_{\Sigma^r_n}$ can be obtained from this variety by an iterated blow-up along the torus-invariant subvarieties $\widehat{H}_{\widetilde{J}}$ associated to the cones $\text{Cone}(v_{H_{\widetilde{J}}})$ for each nested set $\widetilde{J} = (J,\b)$.  Specifically, we have
\[\widehat{H}_{\widetilde{J}} := \bigcap_{i \in J} \widehat{H}_{i}^{-\b(i)},\]
where $\widehat{H}_i^j$ denotes the pullback of $\widehat{H}_j$ along the projection of $X_{(\Sigma_r)^{\times n}}$ to the $i$th factor.  Thus, we have a sequence of blow-ups
\[X_{\Sigma^r_n} = \widehat{Y}_n \longrightarrow \widehat{Y}_{n-1} \longrightarrow \cdots \longrightarrow \widehat{Y}_1 \longrightarrow \widehat{Y}_0 = X_{(\Sigma_r)^{\times n}},\]
where $\widehat{Y}_{k+1}$ is obtained from $
\widehat{Y}_k$ by blow-up along the proper transform of the locus
\[\widehat{W}_k:= \bigcup_{\widetilde{J} \; \big| \; |J| =n-k} \widehat{H}_{\widetilde{J}} \subseteq X_{(\Sigma_r)^{\times n}}.\]
This is exactly analogous to the situation for $\L^r_n$ described above, and also as in that situation, if $\widehat{E}_{\tI} \subseteq \widehat{Y}_{n-|I|+1}$ denotes the exceptional divisor over $\widehat{W}_{\tI}$, then the torus-invariant stratum $X_{\tI}$ is the proper transform of $\widehat{E}_{\tI}$ in $X_{\Sigma^r_n}$.

Now, let
\[i: Y_0 \hookrightarrow \widehat{Y}_0\]
be the linear inclusion of $(\P^1)^n$ into $(\P^{r-1})^n$ sending the $r$th root of unity $\zeta^j$ in each factor to the coordinate hyperplane $\widehat{H}_j$.  Then 
\[W_{\tI} = i^{-1}\left(\widehat{W}_{\tI}\right),\]
so the blow-up closure lemma shows that $Y_{n-|I|+1} \hookrightarrow \widehat{Y}_{n-|I|+1}$ in such a way that $E_{\tI}$ is the restriction of $\widehat{E}_{\tI}$.  Taking proper transforms, then, we see that $D_{\tI}$ is the restriction of $X_{\tI}$, as claimed.
\end{proof}

\begin{remark}
\label{rem:inclusionreversing}
One upshot of Proposition~\ref{prop:strata} is that there is an inclusion-{\it reversing} bijection between the cones of the fan $\Sigma^r_n$ and the boundary strata of $\L^r_n$.  This is analogous to the inclusion-{\it preserving} bijection between the faces of the polytopal complex $\Delta^r_n$ and the boundary strata of $\L^r_n$ that we proved in \cite{CDHLR}.  In Section~\ref{sec:tropical} below, we make the connection between $\Sigma^r_n$ and $\Delta^r_n$ precise.
\end{remark}

\subsection{Calculation of the Chow ring of \texorpdfstring{$\L^r_n$}{the moduli space}}

Equipped with the results of the previous subsections, the calculation of $A^*(\L^r_n)$ is essentially immediate.

\begin{theorem}
\label{thm:LrnChow}
Let $r \geq 2$ and $n \geq 0$.  The Chow ring of $\L^r_n$ is generated by the boundary divisors $D_{\tI}$ for each (nonempty) decorated subset $\tI$ of $[n]$, with relations given by
\begin{itemize}
    \item $D_{\tI} \cdot D_{\widetilde{J}} = 0$ unless either $\tI \leq \widetilde{J}$ or $\widetilde{J} \leq \tI$;
    \item for all $i \in [n]$ and all $a,b \in \Z_r$,
    \[\sum_{\substack{\tI \text{ s.t. }\\ i \in I, \; a(i) = a}} D_{\tI} = \sum_{\substack{\tI \text{ s.t. }\\ i \in I, \; a(i) = b}} D_{\tI}.\]
\end{itemize}
\end{theorem}
\begin{proof}
Theorem~\ref{thm:wcproduct} shows that $\L^r_n \hookrightarrow X_{\Sigma^r_n}$ is a Chow equivalence, and standard toric geometry machinery (see, for example, \cite{CLS}) calculates the Chow ring of $X_{\Sigma^r_n}$.  Namely, it is generated by the torus-invariant divisors, which correspond to the rays of $\Sigma^r_n$ and are thus of the form $X_{\tI}$ for each decorated set $\tI$.  The relations between these generators are given by
\begin{equation}
    \label{eq:relation1}
X_{\widetilde{I_1}} \cdots X_{\widetilde{I_k}} = 0 \; \text{ if } \; \text{Cone}\{\sigma_{\widetilde{I_1}}, \ldots, \sigma_{\widetilde{I_k}}\} \notin \Sigma^r_n
\end{equation}
and
\begin{equation}
    \label{eq:relation2}
\sum_{\tI} \langle v, u_{\tI} \rangle X_{\tI} = 0 \;\; \text{ for all } v \in (V_{\calA})^{\vee},
\end{equation}
where 
\[u_{\tI} = \sum_{i \in I} e_i^{-a(i)}\]
is the primitive integral generator of $\sigma_{\tI}$ and $\langle \cdot, \cdot \rangle$ is the natural pairing between $V_{\calA}$ and $(V_{\calA})^{\vee}$.

By the definition of $\Sigma^r_n$ and the result of Proposition~\ref{prop:strata}, the relation \eqref{eq:relation1} pulls back to
\[D_{\widetilde{I_1}} \cdots D_{\widetilde{I_k}} = 0 \; \text{ if } \; \{\tI_1, \ldots, \tI_k\} \text{ is not a chain},\]
which is equivalent to the first relation in the statement of the theorem.  In the relation \eqref{eq:relation2}, we can let $v$ range over the dual basis to the basis $\{e_i^j\}$ for $V_{\calA}$, where $i \in [n]$ and $j \in \{1,\ldots, r-1\}$; note that in this basis, we have
\[e_i^0 = -e_i^1 - \cdots - e_i^{r-1}\]
by the definition of $V_{\calA}$ as a quotient.  When $v$ is dual to $e_i^j$, the pullback of \eqref{eq:relation2} becomes
\[\sum_{\substack{\tI \text{ s.t. }\\ i \in I, \; a(i) = -j}}  D_{\tI} - \sum_{\substack{\tI \text{ s.t. }\\ i \in I, \; a(i) = 0}} D_{\tI} = 0.\]
Varying over all $v$ in the dual basis yields the second relation in the statement of the theorem, so the theorem is proved.
\end{proof}

\begin{remark}
\label{rem:r=1Chow}
Recalling from Remark~\ref{rem:r=1} that setting $r=1$ in the definition of $\L^r_n$ produces the space $\M^1_n$ considered in Section~\ref{sec:newhassett}, one might hope to generalize Theorem~\ref{thm:LrnChow} to $r=1$ by calculating the Chow ring of $\M^1_n$.  This can indeed be done: by the iterated blow-up perspective described in the proof of Theorem~\ref{thm:Lrnwc}, one can view $\M^1_n$ as the toric variety associated to a fan obtained by stellar subdivision from the fan for $(\P^1)^n$.  This fan is not the $r=1$ case of the nested set fan $\Sigma^r_n$, however, so the Chow ring of $\M^1_n$ does not arise as a special case of Theorem~\ref{thm:LrnChow}.
\end{remark}

\begin{remark}
\label{rem:exceptional-iso}
A further application of the presentation of $\L^r_n$ as a wonderful compactification, which we hope to take up in future work, is a computation of the $K$-ring of $\L^r_n$.  In particular, \cite{LLPP} gives an isomorphism between the integral $K$-ring and the Chow ring of wonderful compactifications of hyperplane arrangement complements in projective space.  If a similar result holds for wonderful compactifications of complements of product arrangements, then the computation of $A^*(\L^r_n)$ in Theorem~\ref{thm:LrnChow} will yield a computation of $K(\L^r_n)$. 
\end{remark}

\subsection{The Betti numbers of \texorpdfstring{$\L^r_n$}{the moduli space}}

While the computation of $A^*(\L^r_n)$ in the previous subsection relies critically on the Chow equivalence with $X_{\Sigma^r_n}$ provided by Theorem~\ref{thm:wcproduct}, one can compute $A^*(\L^r_n)$ as an additive group without passing through that theorem.  Indeed, in \cite{lili:chow}, Li Li gives a presentation of the Chow groups $A^*(\overline{Y}_{\calG})$ for any wonderful compactification $\overline{Y}_{\calG}$.  In the case of $\L^r_n$, that presentation is the following.

First, for any chain $\tbI = (I_1, \ldots, I_\ell, \a)$, set
\[j(\tbI) = (j_1(\tbI), \ldots, j_\ell(\tbI)):= (|I_1|, |I_2|-|I_1|, \ldots, |I_{\ell}|- |I_{\ell-1}|),\]
which we refer to as the {\bf jump type} of $\tbI$.  Then, define
\[M_{\tbI} := \{ \mu \in \Z^{\ell} \; | \; 1 \leq \mu_i < j_i(\tbI) \; \text{ for all } i\}.\]
Note that $M_{\tbI}$ depends only on the jump type of $\tbI$, and it is nonempty if and only if each entry in $j(\tbI)$ is at least two.  In light of this, for any vector $\bf{j}=(j_1, \ldots, j_\ell) \in (\Z_{\geq 2})^{\ell}$, let
\[M_{\bf{j}} := \{ \mu \in \Z^{\ell} \; | \; 1 \leq \mu_i < j_i \; \text{ for all } i\},\]
and let $N_{\bf j}$ be the number of chains of jump type $\bf{j}$; explicitly,
\[N_{\bf j} := \binom{n}{j_1, \ldots, j_\ell, n - |\mathbf{j}|} r^{|\bf{j}|},\]
where $|\bf{j}| := j_1 + \cdots + j_\ell$.\footnote{The published version of this manuscript contains a minor error where the last entry $n - |\mathbf{j}|$ was missing in the multinomial coefficient.} Then the presentation of the Chow groups $A^*(\L^r_n)$ is the following. 

\begin{theorem} For any $k \in \Z^{\geq 0}$, there is an isomorphism of additive groups
\[ A^k(\L^r_n) \cong A^k((\P^1)^n) \oplus \bigoplus_{\substack{\ell \geq 1 \\ \bf{j} \in (\Z_{\geq 2})^{\ell}}} \left(  \bigoplus_{\mu \in M_{\bf{j}}} A^{k-|\mu|}\left((\P^1)^{n-|\bf{j}|}\right) \right)^{\oplus N_{\bf{j}}}.
\]
\end{theorem}
\begin{proof}
This is a direct application of \cite[Theorem 3.1]{lili:chow}.  The sum over $\calG$-nested sets $\mathcal{T}$ in that theorem becomes a sum over chains $\tbI$, and (after correcting the typo that $\{\mu_G\}_{G \in \calG}$ should be $\{\mu_G\}_{G \in \mathcal{T}}$ in \cite[page 9]{lili:chow}) the set $M_{\mathcal{T}}$ becomes the set $M_{\tbI}$.  The space $Y_0\mathcal{T}$ in that theorem is the minimal subvariety (under inclusion) in the chain $\mathcal{T}$, which in our case is
\[H_{\tbI} := \sum_{i \in I_{\ell}} H_i^{\a(i)} \cong (\P^1)^{n-|I_\ell|}.\]
Since $|I_\ell| = |\bf{j}|$ for any chain $\tbI$ of jump type $\bf{j}$, the above isomorphism follows.
\end{proof}

\begin{example}
    Using the theorem above, we compute the following table of Betti numbers of $\L^{r}_n$ for small $r$ and $n$ using SageMath.\footnote{SageMath code available at \url{https://github.com/shiyue-li/multimatroids/blob/main/r-Eulerian.sage}.} Note that the Betti numbers for $r = 2$ are precisely the type-$B$ Eulerian numbers,\footnote{OEIS A060187: \url{https://oeis.org/A060187}.} which were studied as the Betti numbers of the type-$B$ permutohedral variety $X_{B_n} = \L^2_n$ in \cite{eur2022signed}. 
    \begin{center}
    \begin{tabular}{ |c|c| } 
    \hline
    $(r, n)$ & $\dim A^{i} (\L^{r}_n ), i = 0, \ldots, n$ \\ \hline
    $(2, 3)$ & $1, 23, 23, 1$ \\ \hline
    $(2, 4)$ & $1, 76, 230, 76, 1$ \\ \hline
    $(2, 5)$ & $1, 237, 1682, 1682, 237, 1$ \\ \hline
    $(2, 6)$ & $1, 722, 10543, 23548, 10543, 722, 1$ \\ \hline
    $(3, 4)$ & $1, 247, 897, 247, 1$ \\ \hline
    $(3, 5)$ & $1, 1013, 9433, 9433, 1013, 1$ \\ \hline
    $(3, 6)$ & $1, 4083, 82905, 202115, 82905, 4083, 1$ \\ \hline
    $(3, 7)$ & $1, 16369, 663897, 3268709, 3268709, 663897, 16369, 1$ \\ \hline
    $(4, 5)$ & $1, 3109, 34154, 34154, 3109, 1$ \\ \hline
    $(4, 6)$ & $1, 15606, 384719, 988084, 384719, 15606, 1$ \\ \hline
    $(4, 7)$ & $1, 78103, 3939429, 21024707, 21024707, 3939429, 78103, 1$ \\ \hline
    \end{tabular}
    \end{center}
\end{example}

This table supports the following conjecture, the $r=2$ case of which follows from \cite{eur2022signed}. \footnote{Previous versions, including the published version, contained a minor error in the Betti numbers due to the error mentioned in the previous page; this table has been corrected.}

\begin{conjecture}
    For each $r$ and $n$, the Betti numbers $\dim A^{i}(\L^r_n)$ form a log-concave sequence.
\end{conjecture}

\section{Connection to tropical curves with cyclic action}
\label{sec:tropical}

We have now seen that the nested set fan $\Sigma^r_n$ given by Definition~\ref{def:sigma} yields a toric variety whose Chow ring is isomorphic to $A^*(\L^r_n)$.  This fan has another interpretation, however: its support can be identified with the moduli space of ``tropical $(r,n)$-curves,'' and under this identification, the subdivision of $|\Sigma^r_n|$ into cones coincides with the stratification of the tropical moduli space by analogues of boundary strata.  The goal of this section is to prove these assertions.  As a consequence, we also find a new interpretation of the polytopal complex $\Delta^r_n$ introduced in \cite{CDHLR}.

\subsection{The fan \texorpdfstring{$\Sigma^r_n$}{Sigmarn} as the tropical moduli space}

Recall that the dual graph of an element of $\L^r_n$ is a combinatorial graph with a vertex for each irreducible component of the underlying curve, an edge for each node, and a half-edge for each marked point (see \cite[Definition 2.9]{CDHLR}).  If $\Gamma$ is the dual graph of an element $(C; \sigma; x^\pm, \{y^\ell\}, \{z_i^j\})$ of $\L^r_n$, then $\sigma$ induces a unique automorphism  $\sigma_{\Gamma}$ of $\Gamma$.  Given this, tropical $(r,n)$-curves are defined as follows.

\begin{definition}
Let $n \geq 0$ and $r \geq 2$.  A {\bf tropical $(r,n)$-curve} is a triple $(\Gamma,\sigma_\Gamma, L)$, where $\Gamma$ is the dual graph of an element $(C; \sigma; x^\pm, \{y^\ell\}, \{z_i^j\})$ in $\L^r_n$, $\sigma_\Gamma$ is the unique automorphism  on the graph $\Gamma$ determined by $\sigma$, and
\[L: E(\Gamma) \rightarrow \R^+\]
is a ``length'' function on the edges of $\Gamma$ such that
\[L(e) = L(\sigma_{\Gamma}(e))\]
for all $e \in E(\Gamma)$.
\end{definition}

We denote by $\Lrntrop$ the set of all tropical $(r,n)$-curves.  Our goal, now, is to identify $\Lrntrop$ with $|\Sigma^r_n|$.  In particular, the cones of $\Sigma^r_n$ will be identified with subsets of $\Lrntrop$, and in order to do so, we recall from Remark~\ref{rem:inclusionreversing} that the cones of $\Sigma^r_n$ are in inclusion-reversing bijection with the boundary strata $S_{\tbI}$ of $\L^r_n$.  Thus, for any chain $\tbI$, we define $T_{\tbI} \subseteq \Lrntrop$ as the subset consisting of tropical curves $(\Gamma, L)$ where the boundary stratum with dual graph $\Gamma$ contains $S_{\tbI}$.  More explicitly, if $\Gamma_{\tbI}$ denotes the dual graph of a curve of type $\tbI$ (as in Definition~\ref{def:SI}), we have
\[T_{\tbI} := \{ (\Gamma, L) \in \Lrntrop \; | \; \Gamma \text{ is obtained from } \Gamma_{\tbI} \text{ by contracting edges}\}.\]
Given this definition, we can state the correspondence between $\Lrntrop$ and $|\Sigma^r_n|$ as follows.

\begin{proposition}
\label{prop:tropical}
There is a natural bijection between $\Lrntrop$ and $|\Sigma^r_n|$, under which the subset $T_{\tbI}$ corresponds to the cone $\sigma_{\tbI}$.
\end{proposition}
\begin{proof}
Recall from Remark~\ref{rem:stellar} that $\Sigma_r^n$ is obtained by stellar subdivision from the fan $\Sigma_r^{\times n}$, where $\Sigma_r$ is the $1$-dimensional fan in $\R^r/\R$ with $r$ rays, one spanned by the image of each of the standard basis vectors in $\R^r$.  Thus, one has
\begin{equation}
    \label{eq:|Sigma|}
|\Sigma^r_n| =|\Sigma_r|^{\times n}  = \{x_1 e_1^{a_1} + \cdots + x_ne_n^{a_n} \; | \; a_i \in \Z_r, \; x_i \in \R^{\geq 0} \; \text{ for all } i\} \subseteq (\R^r/\R)^{\oplus n}.
\end{equation}

In order to identify $\Lrntrop$ with this set, we associate to each $(\Gamma,L) \in \Lrntrop$ a point in $(\R^r/\R)^{\oplus n}$.  Specifically, let $L_i$ denote the total length of the edges of $\Gamma$ in a path from the central vertex to the vertex containing $z^0_i$, and assuming $L_i \neq 0$, define $\ell_i \in \Z_r$ by the condition that $z^0_i$ is on the same spoke as $y^{\ell_i}$. Then we identify $(\Gamma, L) \in \Lrntrop$ with the point
\[\sum_{i \; | \; L_i \neq 0} L_i e_i^{\ell_i} \in (\R^r/\R)^{\oplus n}.
\]
Given that $L_i$ varies over all nonnegative real numbers and $\ell_i$ varies over all elements of $\Z_r$, the image of $\Lrntrop$ under this identification is precisely the set \eqref{eq:|Sigma|}.

To understand the image of $T_{\tbI}$ under this identification, recall that in the dual graph $\Gamma_{\tbI}$ of a generic element of $S_{\tbI}$, the marked points $z_i^j$ with $i \in I_1$ are on the outermost vertices, so in the image of a tropical curve $(\Gamma_{\tbI}, L)$, the $L_i$ with $i \in I_1$ are equal and largest among all $L_i$.  Similarly, the marked points with $z_i^j$ with $i \in I_2 \setminus I_1$ are on the next-to-outermost vertices, so the $L_i$ with $i \in I_2 \setminus I_1$ are equal and next-largest.  This continues until the marked points $z_i^j$ with $i \in [n] \setminus I_\ell$, which are on the central vertex, so $L_i=0$ for $i \in [n] \setminus I_\ell$.  It follows that the set $T_{\tbI} \subseteq \Lrntrop$ corresponds under the above identification to the set of points
\[L_1 e_1^{-a(1)} + \cdots + L_n e_n^{-a(n)} \in (\R^r/\R)^{\oplus n}\]
for which $L_1, \ldots, L_n \in \R^{\geq 0}$ satisfy the following conditions:
\begin{itemize}
    \item if $i,i' \in I_j \setminus I_{j-1}$ for some $j$, then $L_i = L_{i'}$;
    \item if $i_1 \in I_1, i_2 \in I_2, \ldots, i_\ell \in I_\ell$, then
    \[L_{i_1} \geq L_{i_2} \geq \cdots \geq L_{i_\ell};\]
    \item if $i \in [n] \setminus I_\ell$, then $L_i = 0$.
\end{itemize}
To see that this set coincides with $\sigma_{\tbI}$, recall from Definition~\ref{def:sigma} that \begin{align*}
    \sigma_{\tbI} &:= \text{Cone}\left\{ \sum_{i \in I_1} e_i^{-a(i)}, \ldots, \sum_{i \in I_\ell} e_i^{-a(i)}\right\}\\
&=\left\{\left. c_1\sum_{i \in I_1}e_i^{-a(i)} + \cdots + c_\ell \sum_{i \in I_\ell} e_i^{-a(i)} \; \right| \; c_1, \ldots, c_\ell \in \R^{\geq 0}\right\}.
\end{align*}
Collecting the terms in a different way and using that $I_1 \subseteq I_2 \subseteq \cdots \subseteq I_\ell$, an arbitrary point in $\sigma_{\tbI}$ can be expressed as
\[\sum_{i \in I_1} (c_1 + \cdots + c_\ell)e_i^{-a(i)} + \sum_{i \in I_2 \setminus I_1} (c_2 + \cdots + c_\ell)e_i^{-a(i)} +\cdots + \sum_{i \in I_\ell \setminus I_{\ell-1}} c_\ell \; e_i^{-a(i)}\]
for $c_1, \ldots, c_\ell \in \R^{\geq 0}$.  Thus, the coefficient on $e_i^{-a(i)}$ for any $i \in I_1$ is the same, and these are the largest coefficients; the coefficients on $e_i^{-a(i)}$ for any $i \in I_2 \setminus I_1$ are the same, and these are the next-largest coefficients; and so on.  This is precisely the set of points satisfying the conditions mentioned above, so the identification of $\sigma_{\tbI}$ with $T_{\tbI}$ is complete.
\end{proof}

\begin{remark}
Aside from the definition of $\Lrntrop$ given above, there is another sense in which one might ``tropicalize'' the moduli space $\L^r_n$.  Namely, one can embed $\mathcal{L}^r_n \hookrightarrow \mathbb{T}^r$ as a closed subvariety (as in Remark~\ref{rem:linear}), and as such there is an associated geometric tropicalization $\text{Trop}(\mathcal{L}^r_n)$ in the sense of \cite{HKT}.  To see that these two notions of the tropical moduli space coincide, recall from the proof of Theorem~\ref{thm:wcproduct} that $\L^r_n \subseteq X_{\Sigma^r_n}$ is a tropical compactification, meaning in particular that
\[|\Sigma^r_n| = \text{Trop}(\mathcal{L}^r_n).\]
Combining this with Proposition~\ref{prop:tropical} gives an identification
\[\Lrntrop = \text{Trop}(\mathcal{L}^r_n).\]
\end{remark}

\subsection{The polytopal complex \texorpdfstring{$\Delta^r_n$}{Deltarn} as a normal complex of \texorpdfstring{$\Sigma^r_n$}{Sigmarn}}

The results of the previous subsection generalize the situation for Losev--Manin space $\L_n$, which---as explained in \cite{CDHLR}---is ``morally'' the $r=1$ case of the spaces $\L^r_n$.  In particular, Losev and Manin showed in \cite{LM} that $\L_n$ is a toric variety whose associated fan can be identified with the tropical moduli space $L_n^{\text{trop}}$.  Because this is a complete fan, though, one can also view the connection in terms of polytopes: namely, the normal polytope to $L_n^{\text{trop}}$ is the polytope of $\L_n$ as a toric variety, meaning that its faces are identified with the torus-invariant strata.  In fact, this normal polytope is the $(n-1)$-dimensional permutohedron $\Pi_n$, and the torus-invariant strata are precisely the boundary strata, so one obtains an identification between the faces of $\Pi_n$ and the boundary strata in $\L_n$.

In the case of $\L^r_n$, the moduli space itself is not toric but sits inside of (and is Chow-equivalent to) the toric variety $X_{\Sigma^r_n}$ whose fan we have now identified with $\Lrntrop$.  However, $\Sigma^r_n$ is not a complete fan in $(\R^r/\R)^{\oplus n}$ for $r >2$, so the usual construction of the normal polytope does not apply; it produces a polytope, but one of larger dimension than $|\Sigma^r_n|$.  There is a substitute for the normal polytope for non-complete fans, though, which is the ``normal complex'' introduced by Nathanson--Ross \cite{NR}.  This is a polytopal complex that one can view as the result of truncating $\Sigma^r_n$ by normal hyperplanes.  To complete the analogy to Losev--Manin space, then, one would hope to identify the faces of this normal complex---for an appropriate interpretation of ``faces'' of a polytopal complex---with the boundary strata in $\L^r_n$. 

In our previous work \cite{CDHLR}, we have already identified the boundary strata in $\L^r_n$ with the ``$\Delta$-faces'' of another polytopal complex $\Delta^r_n$.  This polytopal complex was constructed as a subset of
\[(\R^{\geq 0} \cdot \mu_r)^n \subseteq \C^n,\]
where $\mu_r$ denotes the set of $r$th roots of unity.  However, we can identify
\[(\R^{\geq 0} \cdot \mu_r)^n \leftrightarrow |\Sigma^r_n|\]
by identifying
\[\left(x_1 \zeta^{a_1},  \ldots, x_n \zeta^{a_n}\right) \leftrightarrow x_1 e_1^{a_1} + \cdots + x_n e_n^{a_n},\]
and using this, we can view $\Delta^r_n$ as a subset of $(\R^r/\R)^{\oplus n}$.  Explicitly,
\begin{equation}
    \label{eq:Deltarn}
\Delta^r_n := \bigcup_{a_1, \ldots, a_n \in \Z_r} \left\{ x_1e_1^{a_1} + \cdots + x_n e_n^{a_n} \; \left| \; x_i \in \R^{\geq 0} \; \text{ for all } i, \; \sum_{i \in I} x_i \leq \delta^n_{|I|} \; \text{ for all } I \subseteq [n]\right. \right\},
\end{equation}
where
\[\delta^n_k:= n+(n-1) + (n-2) + \cdots + (n-k+1).\]

We claim that this complex $\Delta^r_n$ is the normal complex of $\Sigma^r_n$.  (In the case $r=n=2$, the fan $\Sigma^2_2$ is the complete fan shown in Figure~\ref{fig:stellar}, whose normal complex is in fact a normal polytope: the octagon, which is the signed permutohedron when $n=2$ and equals $\Delta^2_2$.  In the case $r=3$ and $n=2$, we illustrate the claim in Figure~\ref{fig:stellar32}.)

\begin{figure}[ht]
    \centering
    \definecolor{ao(english)}{rgb}{0.0, 0.5, 0.0}

\begin{tikzpicture}[xscale=0.23, yscale=0.23]
\draw [gray,fill=gray,opacity=0.3] (0,0) -- (-8/2,-7/2) -- (-8/2+7*0.7,-7/2-5*0.7) -- (7*0.7,-5*0.7) -- cycle;
\draw [gray,fill=gray,opacity=0.3] (0,0) -- (7,0) -- (7+7*0.7,-5*0.7) -- (7*0.7,-5*0.7) -- cycle;
\draw [gray,fill=gray,opacity=0.3] (0,0) -- (0,7) -- (7*0.7,7-5*0.7) -- (7*0.7,-5*0.7) -- cycle;
\filldraw (0,0) circle[radius=1.5mm];
\draw [very thick,->] (0,0) --(0,8);
\draw [very thick,->] (0,0) --(-8*0.8,-7*0.8);
\draw [very thick,->] (0,0) --(10,0);
\draw [very thick,->] (0,0) --(7,-5);
\node at (1, -10) {$e_1^0 \times \Sigma_3$};
\node [above, scale=0.7] at (0,8) {$0 \times e^1_2$};
\node [left,scale=0.7] at (-8*0.8,-7*0.8) {$0\times e^2_2$};
\node [right,scale=0.7] at (10,0) {$0 \times e^0_2$};
\node [below right,scale=0.7] at (7,-5) {$e_1^0 \times 0$};
\node [scale=0.7] at (0,-3) {$e_1^0\times e^2_2$};
\node [scale=0.7] at (2.5,3) {$e_1^0 \times e^1_2$};
\node [scale=0.7] at (7,-2) {$e_1^0 \times e^0_2$};
\begin{scope}[shift={(21,0)}]
\draw [very thick,->] (0,0) --(0,8);
\draw [very thick,->] (0,0) --(-8*0.8,-7*0.8);
\draw [very thick,->] (0,0) --(10,0);
\draw [very thick,->] (0,0) --(7,-5);
\draw [blue, very thick,->] (0,0) --(7+7*0.7,-5*0.7);

\draw [fill=blue!60!cyan,opacity=0.3] (0,0) -- (7*0.7*0.8,-5*0.7*0.8) -- (7*0.8+7*0.7*0.8,-5*0.7*0.8) -- cycle;
\draw [fill=blue,opacity=0.3] (0,0) -- (7*0.8,0) -- (7*0.8+7*0.7*0.8,-5*0.7*0.8) -- cycle;

\draw [ao(english),fill=ao(english),opacity=0.3] (0,0) -- (-8/2*0.8,-7/2*0.8) -- (-8/2*0.8+7*0.7*0.8,-7/2*0.8-5*0.7*0.8) -- cycle;
\draw [ao(english),fill=ao(english)!60!green,opacity=0.3] (0,0) -- (-8/2*0.8+7*0.7*0.8,-7/2*0.8-5*0.7*0.8) -- (7*0.7*0.8,-5*0.7*0.8) -- cycle;

\draw [magenta,fill=magenta,opacity=0.3] (0,0) -- (0,7*0.8) -- (7*0.7*0.8,7*0.8-5*0.7*0.8)  -- cycle;
\draw [magenta,fill=magenta!60!pink,opacity=0.3] (0,0) -- (7*0.7*0.8,-5*0.7*0.8) -- (7*0.7*0.8,7*0.8-5*0.7*0.8) -- cycle;

\filldraw (0,0) circle[radius=1mm];

\draw[ao(english), very thick,->]  (0,0) --(-8/2+7*0.7,-7/2-5*0.7);
\draw [magenta, very thick,->]  (0,0) --(7*0.7,7-5*0.7);
\filldraw (0,0) circle[radius=1.5mm];

\node at (1, -10) {portion of $\Sigma^3_2$};
\end{scope}
\begin{scope}[shift={(42,0)}]
\draw [->] (0,0) --(0,8);
\draw [->] (0,0) --(-8*0.8,-7*0.8);
\draw [->] (0,0) --(10,0);
\draw [->] (0,0) --(7,-5);
\draw [blue, ->] (0,0) --(7*0.8+7*0.56,-5*0.56);
\draw[ao(english), ->]  (0,0) --(-8/2*0.8+7*0.56,-7/2*0.8-5*0.56);
\draw [magenta, ->]  (0,0) --(7*0.7,7-5*0.7);

\draw [yellow,fill=yellow,opacity=0.3] (0,0) -- (0,7*0.8) -- (7*0.37*0.8,7*0.8-5*0.37*0.8) -- (7*0.7*0.8,7*0.8-5*1.2*0.8) -- (7*0.7*0.8,-5*0.7*0.8) ;

\draw [yellow,fill=yellow,opacity=0.3] (0,0) -- (-8/2*0.8,-7/2*0.8)-- (-8/2*0.8+7*0.37*0.8,-7/2*0.8-5*0.37*0.8)--  (-8/2*0.8*0.5+7*0.7*0.8,-7/2*0.8*0.5-5*0.7*0.8) -- (7*0.7*0.8,-5*0.7*0.8);

\draw [yellow,fill=yellow,opacity=0.3] (0,0) -- (7*0.8,0) -- (7*0.8+7*0.37*0.8,-5*0.37*0.8)--  (7*0.8+7*0.27*0.8,-5*0.7*0.8) -- (7*0.7*0.8,-5*0.7*0.8);

\draw [very thick] (0,7*0.8) -- (7*0.37*0.8,7*0.8-5*0.37*0.8) -- (7*0.7*0.8,7*0.8-5*1.2*0.8) -- (7*0.7*0.8,-5*0.7*0.8) ;

\draw [very thick] (-8/2*0.8,-7/2*0.8)-- (-8/2*0.8+7*0.37*0.8,-7/2*0.8-5*0.37*0.8)--  (-8/2*0.8*0.5+7*0.7*0.8,-7/2*0.8*0.5-5*0.7*0.8) -- (7*0.7*0.8,-5*0.7*0.8);

\draw [very thick] (7*0.8,0) -- (7*0.8+7*0.37*0.8,-5*0.37*0.8)--  (7*0.8+7*0.27*0.8,-5*0.7*0.8) -- (7*0.7*0.8,-5*0.7*0.8);

\draw [very thick, magenta] (7*0.37*0.8,7*0.8-5*0.37*0.8) -- (7*0.7*0.8,7*0.8-5*1.2*0.8);
\draw [very thick, ao(english)] (-8/2*0.8+7*0.37*0.8,-7/2*0.8-5*0.37*0.8)--  (-8/2*0.8*0.5+7*0.7*0.8,-7/2*0.8*0.5-5*0.7*0.8);
\draw [very thick, blue](7*0.8+7*0.37*0.8,-5*0.37*0.8)--  (7*0.8+7*0.27*0.8,-5*0.7*0.8);

\filldraw [thin, blue, fill= blue!50!white] (7*0.8+7*0.37*0.8,-5*0.37*0.8) circle[radius=1.5mm];
\filldraw [thin, blue, fill=blue!50!cyan] (7*0.8+7*0.27*0.8,-5*0.7*0.8) circle[radius=1.5mm];
\filldraw [thin, ao(english),fill=ao(english)!50!white] (-8/2*0.8+7*0.37*0.8,-7/2*0.8-5*0.37*0.8) circle[radius=1.5mm];
\filldraw [thin, ao(english),fill=ao(english)!50!green](-8/2*0.8*0.5+7*0.7*0.8,-7/2*0.8*0.5-5*0.7*0.8) circle[radius=1.5mm];
\filldraw [thin, magenta,fill=magenta!50!white](7*0.37*0.8,7*0.8-5*0.37*0.8) circle[radius=1.5mm];
\filldraw [thin, magenta,fill=magenta!50!pink] (7*0.7*0.8,7*0.8-5*1.2*0.8) circle[radius=1.5mm];
\node at (1, -10) {portion of $\Delta^3_2$};

\end{scope}
\end{tikzpicture}
    \caption{The fan $\Sigma_2^3$ is obtained as stellar subdivision of $\Sigma_3 \times \Sigma_3$, so we can obtain a portion of it by stellar subdivision of $e_1^0 \times \Sigma_3$, as shown in the middle figure.  Taking the dual complex to this fan, we recover a portion of the complex $\Delta^3_2$ illustrated in \cite[Figure 2]{CDHLR}.}
    \label{fig:stellar32}
\end{figure}
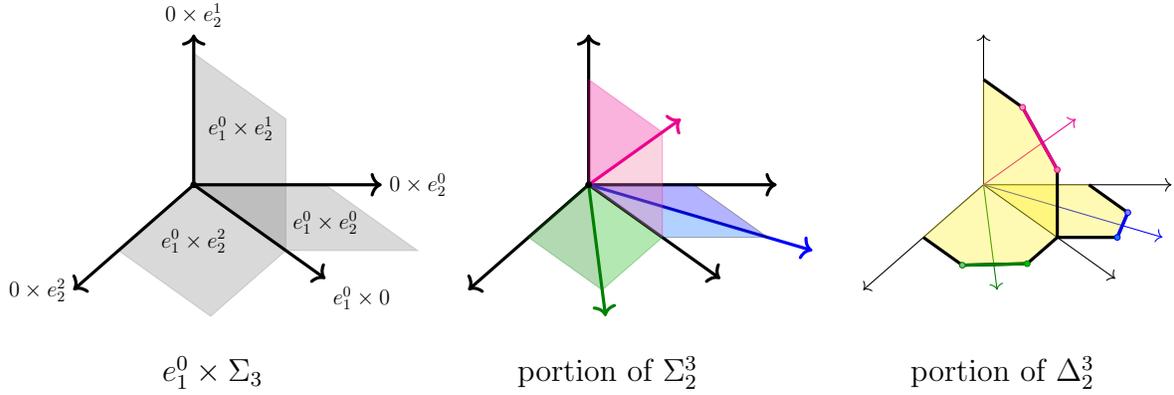

More precisely, normal complexes of fans depend on three choices: an inner product on the ambient vector space, a vector $\vec{z} \in \R^{\Sigma^r_n(1)}$, and a distinguished generator $u_{\rho}$ of each ray $\rho \in \Sigma^r_n(1)$.  The inner product in our case is the dot product on $(\R^r/\R)^{\oplus n}$ in the basis $\{e_i^j\}_{i \in [n], j \in [r-1]}$, which we denote by $\ast$.  As for the vector $\vec{z}$, since the rays of $\Sigma^r_n$ are the cones $\sigma_{\tI}$ for each decorated set $\tilde{I}$, we can define $\vec{z}$ by setting
\begin{equation}
\label{eq:zI}
z_{\tilde{I}} := \delta^n_{|I|}
\end{equation}
for each decorated set $\tI$.  Finally, for the generator of the ray associated to $\tI$, we choose
\begin{equation}
\label{eq:uI}
u_{\tI}:= \sum_{i \in I} e_i^{-a(i)}.
\end{equation}
Equipped with this notation, the final perspective we present on the fan $\Sigma^r_n$ is the following.

\begin{proposition}
\label{prop:normalfan}
The polytopal complex $\Delta^r_n$ is the normal complex of the fan $\Sigma^r_n$ with respect to the inner product $\ast$, the vector $\vec{z}$ defined by \eqref{eq:zI}, and the ray generators defined by \eqref{eq:uI}.
\end{proposition}
\begin{proof}
To define the normal complex of $\Sigma^r_n$, one first truncates all faces by normal hyperplanes.  Explicitly, for each face $\sigma_{\tbI}$ of $\Sigma^r_n$, let
\[P_{\tbI}:= \sigma_{\tbI} \cap \{v \in (\R^r/\R)^{\oplus n} \; | \; v \ast u_{\tbI} \leq z_{\tbI} \; \text{ for all } \rho \in \sigma_{\tbI}(1)\}.\]
Then the normal complex, by definition, is the union of all faces of the polytopes $P_{\tbI}$, over all cones $\sigma_{\tbI}$ of $\Sigma^r_n$.  Because we include faces in this union, it suffices to consider only maximal cones, which are those associated to chains $\tbI = (I_1, \ldots, I_n; \a)$ of length $n$.  For such chains, we have
\[\sigma_{\tbI} = \{x_1 e_1^{-\a(1)} + \cdots + x_n e_n^{-\a(n)} \; | \; x_i \in \R^{\geq 0} \; \text{ for all } i\},\]
and the rays $\rho \in \sigma_{\tbI}(1)$ are the cones generated by $u_{(I_j, \a|_{I_j})}$ for $j \in [n]$.  Thus, the inequalities in the definition of $P_{\tbI}$ amount to the condition that
\[\sum_{i \in I_j} x_i \leq \delta^n_{|I_j|}\]
for all $j \in [n]$.  As $\tbI$ ranges over all maximal chains, the exponents $\a(i)$ range over all elements of $\Z_r$ and the sets $I_j$ range over all subsets of $[n]$, so the normal complex precisely coincides with the set $\Delta^r_n$ of \eqref{eq:Deltarn}.
\end{proof}

\newpage

\bibliographystyle{alpha}
\bibliography{bibliography.bib}

\end{document}